\documentclass[11pt,letterpaper,reqno]{amsart}

\usepackage{amsmath}
\usepackage{amsthm}
\usepackage{amsfonts, mathrsfs}
\usepackage{mathtools}
\usepackage{amssymb}
\usepackage[makeroom]{cancel}
\usepackage{graphicx,caption}
\usepackage{tikz}
\usepackage{tikz-cd}
\usetikzlibrary{shapes,arrows}
\allowdisplaybreaks

\usepackage[skip=4pt plus1pt, indent=0pt]{parskip}

\usepackage{comment}

\usepackage{enumitem}
\usepackage{hyperref}
\hypersetup{bookmarksdepth=4}
\hypersetup{pdfcreator=XeLaTeX with hyperref}
\usepackage{bbm}
\hypersetup{
    colorlinks,
    citecolor=teal,
    filecolor=blue,
    linkcolor=black,
    urlcolor=teal
}
\usepackage{xcolor}

\numberwithin{equation}{section}
\newtheorem{theorem}{Theorem}[section]
\newtheorem{cor}{Corollary}[section]
\newtheorem{lemma}{Lemma}[section]
\newtheorem{prop}{Proposition}[section]

\theoremstyle{definition}
\newtheorem{defin}{Definition}[section]
\newtheorem{remark}{Remark}[section]
\newtheorem{remarks}{Remarks}[section]

\newtheorem*{assum*}{\assumptionnumber}
\providecommand{\assumptionnumber}{}
\makeatletter
\newenvironment{assum}[1]
 {%
  \renewcommand{\assumptionnumber}{Assumption #1}%
  \begin{assum*}%
  \protected@edef\@currentlabel{#1}%
 }
 {%
  \end{assum*}
 }
\makeatother

\newcommand{\vertiii}[1]{{\left\vert\kern-0.25ex\left\vert\kern-0.25ex\left\vert #1 
    \right\vert\kern-0.25ex\right\vert\kern-0.25ex\right\vert}}

\calclayout

\usepackage{setspace}

\newcommand{\bj}{\mathbf{j}}
\newcommand{\bi}{\mathbf{i}}

\usepackage[square,numbers]{natbib}
\renewcommand{\cite}{\citep}

\begin{document}

\title[Opinion dynamics on networks with community structure]{Opinion dynamics on non-sparse networks \\  with community structure}

\author{Panagiotis Andreou}
\address{\textsc{University of North Carolina at Chapel Hill}}
\curraddr{}
\email{pandreou@email.unc.edu}
\thanks{}

\author{Mariana Olvera-Cravioto}
\address{\textsc{University of North Carolina at Chapel Hill}}
\curraddr{}
\email{molvera@email.unc.edu}
\thanks{}

\subjclass[2020]{60G10, 05C80, 60G17, 60J20, 60J85, 60K35}

\keywords{opinion dynamics, inhomogeneous random graphs, interacting particle systems, mean-field approximation, propagation of chaos, strong couplings}

\date{}

\dedicatory{}

\begin{abstract}
We study the evolution of opinions on a directed network with community structure. Individuals update their opinions synchronously based on a weighted average of their neighbors' opinions, their own previous opinions, and external media signals. Our model is akin to the popular Friedkin-Johnsen model, and is able to incorporate factors such as stubbornness, confirmation bias, selective exposure, and multiple topics, which are believed to play an important role in the formation of opinions. Our main result shows that, in the large graph limit,  the opinion process concentrates around its mean-field approximation for any level of edge density, provided the average degree grows to infinity. Moreover, we show that the opinion process exhibits propagation of chaos. We also give results for the trajectories of individual vertices and the stationary version of the opinion process, and prove that the limits in time and in the size of the network commute. The mean-field approximation is explicit and can be used to quantify consensus and polarization. 
\end{abstract}

\maketitle

\section{Introduction}\label{intro}

As the consumption of information through virtual social networks has grown, the number of news outlets has multiplied, and individuals increasingly find ways to disseminate their opinions through online forums and video sharing platforms, populations around the world have become heavily polarized on a number of cultural and political issues \cite{VANBAVEL2021913,Bessi_etal_16,Schmidt_etal_17,Stroud_08}. The asymmetry of the information from one source to another has solidified positions into ideological bubbles with little in common with each other. The formation of such bubbles occurs through a combination of cognitive biases that allow individuals to choose to listen only to opinions that conform to their preconceptions, as well as reject people who espouse views contrary to their own. 

Our collective need to understand the mechanisms behind polarization and propose ideas to reduce its negative effects has inspired much of the current literature in the social sciences \cite{VANBAVEL2021913,Bessi_etal_16,Schmidt_etal_17,Stroud_08, abramowitz2008polarization, gentzkow2016polarization, hetherington2009putting}. Simultaneously, researchers in the mathematical and physical sciences have proposed models that can capture and explain phenomena such as consensus and polarization, often focusing on one or two modeling features at a time. In this work, we analyze a model based on the classical DeGroot \cite{degroot1974reaching} and Friedkin-Johnsen \cite{friedkin1990social} models, that can simultaneously capture features such as:
\begin{enumerate}
\item {\em Confirmation bias:} the tendency to search for, interpret, favor, and recall information in a way that confirms or supports one's prior beliefs or values;
\item {\em Selective exposure:} individuals' tendency to favor information which reinforces their pre-existing views while avoiding contradictory information;
\item {\em Community structure:} when a network divides naturally into groups of nodes with dense connections internally and sparser connections between groups;
\item {\em Bots:} software applications that run automated tasks over the Internet on a large scale, usually with the intent of imitating human activity; 
\item {\em Multiple topics:} the analysis of the evolution of opinions on multiple (possibly correlated) topics at a time. 
\end{enumerate}

A single-topic version of this model was recently analyzed in \cite{fraiman2022opinion} on general complex networks without community structure. This work focused on sparse (i.e., bounded degree) graphs, where the graph's local behavior determines the evolution of the opinion process. Its  main results show the power that the media has in framing people's opinions, and how selective exposure alone can explain much of the polarization that we observe in many societies today. One of the main features of the model in \cite{fraiman2022opinion}, as well as of the model studied here, is that, due to its linearity, it produces computable formulas that can be used to quantify the effect of various model parameters on the evolution of opinions, and potentially, of any intervention one may want to propose. 

The current work studies a more general opinion model than the one analyzed in \cite{fraiman2022opinion} on sparse graphs,  but rather than focus on its explanatory/predictive power, it computes its limiting behavior on denser graphs with community structure. From a modeling point of view, it is often hard to decide whether a real-world network is sparse with a large average degree, or, say, semi-sparse, with an average degree that grows slowly with the size of the graph. Given the lack of precise boundaries, we believe it is interesting to understand the difference between sparse approximations and dense, or mean-field, approximations. At a high level, sparse approximations are determined by the entire local neighborhood of each vertex, while mean-field approximations are obtained by replacing the individual contribution of each neighbor with their average behavior, which is justified only when vertices in the graph have a sufficiently large number of neighbors.

Our main interest in this paper is to show that there is a single mean-field approximation for our opinion model that works for graphs with any density level other than sparse. However, the precision level of this approximation does exhibit a phase transition when the average degree crosses the familiar connectivity threshold $\log n$ (with $n$ the number of vertices in the graph), and seamlessly explains how the model's behavior transitions into the sparse regime, where local approximations such as those found in \cite{fraiman2022opinion} take over. In fact, our analysis of the semi-sparse regime is based on the same local approximation that is used for sparse graphs, which when the average degree grows with $n$ further simplifies to the mean-field approximation obtained in the dense case. The transition into the sparse regime can also be seen from the point of view of the trajectories for different individuals, which in the setting of this paper exhibit propagation of chaos (i.e., asymptotically independent trajectories). In contrast, on sparse graphs, only a weaker form of propagation of chaos holds, which states that only finite sets of uniformly chosen trajectories are asymptotically independent.  The in-depth analysis of the model's explanatory and predictive power will be covered in a separate set of work. 

This paper is organized as follows. In Section~\ref{model} we introduce our model for the evolution of opinions on a large complex network with community structure, as well as the random graph family used in our analysis. In Section~\ref{main_results} we present our main results, which give an approximation for the stationary opinion process that is valid for any density level of the underlying graph other than sparse. In Section~\ref{lit-review} we discuss related models for opinion dynamics, as well as the technical aspects of our work, such as mean-field approximations and local approximations for random graphs. Finally, in Section~\ref{S.Proofs} we give all the proofs for the theorems in Section~\ref{main_results}.

\section{The Model}\label{model}

Our opinion model consists of two parts: 1) the opinion process defined on a fixed directed graph, and 2) the random graph model used to represent the underlying social network. 

\subsection{The opinion process on a fixed graph}

Consider a directed graph $G  = (V, E)$, where individuals in the social network are represented by vertices in the set $V$, and a directed edge $(i,j) \in E$ means that individual $j$ listens to the opinion of individual $i$; $|V|$ denotes the cardinality of the set. Each individual in the set $V$ is assumed to belong to one of $K$ communities, and has a set of attributes that will remain fixed throughout the evolution of the opinion process. We assume we have $\ell$ topics, with $\ell$ a finite number, and individuals' opinions on each topic take values on the interval $[-1,1]$. The main object of study is the opinion process $\{R^{(k)} : k \geq 0\}$, taking values on $[-1,1]^{|V| \times \ell}$, where $R_{ij}^{(k)}$ denotes the expressed opinion of individual $i \in V$ on topic $j \in \{1, \dots, \ell\}$ at time $k \in \mathbb{N} =  \{0, 1, 2, \dots\}$. We will use $\mathbf{R}_i^{(k)}$ to denote the $i$th row of $R^{(k)}$, which corresponds to the multi-topic opinion vector of individual $i$ at time $k$. 

Before we describe the updating rule for the process $\{ R^{(k)}: k \geq 0\}$, we need to describe the vertex attributes that determine how each individual engages with its neighbors and incorporates their own prejudices. To start, each individual $i \in V$ has an attribute vector and an extended attribute vector of the form:
$$\mathbf{a}_i = (J_i, \mathbf{q}_i   )  \in \{ 1, \dots, K\} \times [-1,1]^{\ell} \quad \text{and} \quad \mathbf{y}_i = (J_i, \mathbf{q}_i , \mathbf{b}_i  )  \in \{ 1, \dots, K\} \times [-1,1]^{\ell} \times \mathbb{R}_+^{|V|},$$
where $J_i$ denotes the community to which individual $i$ belongs, $\mathbf{q}_i = (q_{i1}, \dots q_{i\ell})$ is a row vector of pre-conceived ideas or internal beliefs on each of the $\ell$ topics, and $\mathbf{b}_i = (b_{i1}, \dots, b_{i|V|})$ is a nonnegative row vector of ``unnormalized'' weights, where $b_{ij}$ is the weight that individual $i$ gives to individual $j$'s expressed opinions (these weights will later be normalized to add up to one). The extended attribute vectors $\{\mathbf{y}_i: i \in V\}$ do not change throughout the evolution of the process, reflecting the reasonable modeling assumption that they change at a much longer time-scale than the expressed opinions process. To make them easy to identify, all vectors are written in boldface. 

Given the set of edges $E$, we now construct the (normalized) weights according to:
\begin{equation} \label{eq:CommunityMatrix}
c_{ij} = \frac{b_{ij} 1(j \to i)}{\sum_{l \in V} b_{il} 1(l \to i)} 1\left(\sum_{l \in V} b_{il} 1(l \to i) > 0, \,  i\neq j \right), \qquad i,j \in V,
\end{equation}
where $j \to i$ indicates that $(j,i) \in E$. The model also requires two parameters, $c, d \in [0,1]$, such that $0 < c+d \leq 1$ and $d > 0$. The parameter $c$ controls the amount of value that individuals give to the social network, and $d$ controls the amount of weight they give to external influences, e.g., media and political leaders. Although one can easily make the parameters $c$ and $d$ depend on each individual, for simplicity we interpret them as average weights and keep them the same across all individuals. Our assumption that $d > 0$ is important both to the explanatory/predictive power of the model, and to its mathematical properties, as we will see later. 

Having described the attributes of each individual, we can now describe the evolution of the opinion process. At time zero, we initialize the process with a matrix $R^{(0)} \in [-1,1]^{|V| \times \ell}$ chosen according to some initial distribution $\mu_0$ such that $\{\mathbf{R}_i^{(0)} : J_i = r\}$ are identically distributed for each $1\leq r \leq K$. At each time step $k$, each individual $i$ listens to an external signal $\mathbf{W}_i^{(k)} \in [-1,1]^\ell$ (row vector), and updates their opinion according to:
\begin{equation}\label{eqn:recursion_of_opinions_autoregressive}
\mathbf{R}_i^{(k+1)} = c  \sum_{j \in V} c_{ij} \mathbf{R}_j^{(k)} + \mathbf{W}_i^{(k+1)} + (1-c-d) \mathbf{R}_i^{(k)}, \qquad k \geq 0, \, i \in V.
\end{equation}

The external signals $\{ \mathbf{W}_i^{(k)}: i \in V, k \geq 0\}$ are assumed to be independent of each other, and of the initial opinion matrix $R^{(0)}$. In order to include the possibility of having vertices with no inbound neighbors, it is convenient to assume that the external signals take the form:
\begin{equation} \label{eq:MediaSignals}
\mathbf{W}_i^{(k)} = d \mathbf{Z}_i^{(k)} + c \mathbf{q}_i 1\left( d_i^- = 0\right),
\end{equation}
where $d_i^- = \sum_{j \in V} 1( j \to i)$ is the number of inbound neighbors of vertex $i$ and $\{ \mathbf{Z}_i^{(k)}: k \geq 0\}$ can be interpreted as media signals. This form of the external signals replaces the contribution of the neighbors with the vertex's internal belief when it has no inbound neighbors. For a vertex $i \in V$ having attribute vector $\mathbf{a}_i$, the media signals $\{ \mathbf{Z}_i^{(k)}: k \geq 0\}$ are assumed to be i.i.d.~with some distribution $\nu(\mathbf{a}_i)$. If needed, one could make $\nu$ depend on the extended attribute vector of each vertex, and even add more vertex properties to the extended attribute vector, but for simplicity we omit this extension here. We use $W^{(k)}$ to denote the $|V| \times \ell$ matrix whose $i$th row is $\mathbf{W}_i^{(k)}$. 

Under the assumption that $d > 0$, we will see that the process $\{ R^{(k)}: k \geq 0\}$ has a stationary distribution, i.e., $R^{(k)} \Rightarrow R$ as $k \to \infty$, where $\Rightarrow$ denotes weak convergence on $[-1,1]^{|V| \times \ell}$. This paper focuses on the behavior of both the trajectories $\{ \mathbf{R}_i^{(k)}: k \geq 0\}$ and the stationary opinion matrix $R$. 

\subsection{The network}\label{subsection_network}

Although the opinion process itself is defined on a fixed directed graph, our results are based on analyzing the process on a random graph sequence $\{ G_n: n \geq 1\}$, where $G_n = (V_n, E_n)$. From a modeling point of view, one can think of the fixed social network of interest as simply one realization from the random graph model being used. For simplicity, we assume $|V_n| = n$. 

In order to model the community structure and individual-level confirmation bias, we use as the basis for our model a directed stochastic block model (dSBM) with $K$ communities. Each individual in $i \in V_n$ has a community label $J_i$, and given the community labels $\mathscr{J}_n = \{ J_i: i \in V_n\}$,  edge $(i,j)$ is present in the graph with probability:
\begin{equation} \label{eq:EdgeProb}
p_{ij}^{(n)} = \frac{\kappa(J_i, J_j) \theta_n}{n} \wedge 1, \qquad i,j \in V_n,
\end{equation}
independently of everything else, where $\kappa$ is a nonnegative kernel (equivalently, a nonnegative matrix in $\mathbb{R}^{K \times K}$), and $\theta_n$ is a density parameter. 

Since our interest in this paper is to obtain approximations for semi-sparse to dense graphs, it is not important to incorporate degree-corrections into our model other than those produced by the kernel $\kappa$. In the sparse setting, where one can closely model the inhomogeneity of the network (e.g., scale-free degree distributions), using a degree-corrected dSBM would be more appropriate. 

The proportion of vertices belonging to each community is controlled by the probability vector $(\pi_1^{(n)}, \dots, \pi_K^{(n)})$, so that a proportion $\pi_r^{(n)}$ of the vertices in $V_n$ belong to community $r$. The main assumption on the distribution of the community labels is given below. 

\begin{assum}{A}{}  \label{regularity_condition_communities}
There exists a probability vector $\boldsymbol{\pi} = (\pi_1, \dots, \pi_K)$ with $\pi_r > 0$ for all $1\leq r \leq K$, such that 
$$\pi_r^{(n)} = \frac{1}{n} \sum_{i=1}^n 1( J_i = r) \xrightarrow{P} \pi_r,$$
as $n\to\infty$,
for each $r \in \{1, \dots, K\}$. 
\end{assum}

Next, for each $r \in \{1, \dots, K\}$ let $F_r$ be a distribution on $[-1,1]^\ell$. For each vertex $i \in V_n$ sample its internal belief vector $\mathbf{Q}_i$ according to $F_{J_i}$, where $J_i$ is the community label of vertex $i$. 

Finally, to complete the description of the model on a random graph, we sample each of the unnormalized weights $\{ B_{ij}: i,j \in V_n\}$ conditionally independent of each other given the community labels $ \{ J_i: i \in V_n\}$, with $B_{ij}$ sampled from a distribution $G_{J_i,J_j}$ on the interval $[0, H]$, with $H$ a finite constant. Note that the distribution of the unnormalized weights is allowed to depend on the communities to which each vertex belongs, which means one can give less weight to opinions coming from members of a different community. This is a form of \textit{confirmation bias} \cite{mynatt1977confirmation, nickerson1998confirmation, del2017modeling}, applied to the individual rather than the opinion itself. 

\begin{remarks}
We end this section by explaining how one can choose the various parameters to model many of the features that social scientists believe influence the evolution of opinions. Since the full explanatory and predictive power of the model will be discussed in a future companion paper, we give here only a high-level description of what our model can be used for.  
\begin{enumerate}
\item The kernel $\kappa$ and the unnormalized weights $\{ B_{ij}: i,j \in V_n\}$ can be used to model a form of confirmation bias that affects the individuals directly, i.e., by both controlling the number of neighbors that an individual has who belong to communities different from their own, and by discounting their opinions. Note that this form of confirmation bias depends on neither the topic nor the specific expressed opinions being shared at any point. Other ways of modeling confirmation bias directly applied to the expressed opinions being shared are discussed in Section~\ref{lit-review}.

\item A direct way in which our model incorporates selective exposure is through the choice of the distribution $\nu(\mathbf{A}_i)$ for the media signals, where $\mathbf{A}_i = (J_i, \mathbf{Q}_i)$. For example, if we believe that each community has its own preferred media outlets, e.g., news and political commentators, then we can choose the distribution $\nu$ to depend on the community label $J_i$. Moreover, if we want to model the human tendency to reject information that contradicts our internal beliefs or pre-existing views, we can do so by allowing $\nu$ to depend on the internal belief vector $\mathbf{Q}_i$ (e.g., a person who listens to an outlet that has extreme views on topics 1 and 2, but only agrees with the outlet on topic 1, may choose to ignore the media signal for topic 2). 

\item The dSBM can also be used to model the presence of bots, targeted or not, by creating additional communities. A targeted bot chooses who to send its messages to based on their community label. Since bots are not real people, they are not influenced by what they hear from their inbound neighbors, so we can model them to have no inbound neighbors by adjusting the kernel $\kappa$ accordingly. 

\end{enumerate}
\end{remarks}

\section{Main Results}\label{main_results} 

We are now ready to present the main results of the paper. We assume from here onwards that we are working on a probability space $(\Omega, \mathcal{F}, P)$ where we can construct the entire random graph sequence $\{ G_n: n \geq 1\}$ along with their vertex marks, as well as all the external signals needed to define the opinion process on each of the graphs in the sequence. We use $E[\cdot]$ to denote the corresponding expectation. The opinion process $\{ R^{(k)}: k \geq 0\}$ is constructed on a fixed random graph $G_n = (V_n, E_n)$ from the sequence, but for simplicity the dependence on $n$ is not explicit in the notation.

As mentioned earlier, our work shows that there is a unique approximation for the opinion process $\{ R^{(k)}: k \geq 0\}$ under the large graph limit $n \to \infty$, that works for all choices of the density parameter $\theta_n$. The phase transition occurs only in the precision level of the approximation, around the threshold $\theta_n = \log n$. Interestingly, the proof technique is also significantly different above and below the threshold $\log n$, reflecting the fact that as the density parameter $\theta_n$ drops below the threshold, a vertex's entire neighborhood becomes increasingly important. To be able to drop the minimum in \eqref{eq:EdgeProb} and simplify our results, we assume throughout the paper that 
$$\limsup_{n \to \infty} \max_{1 \leq r,s \leq K}  \frac{\kappa(r,s) \theta_n}{n} \leq 1.$$

The process $\{ R^{(k)}: k \geq 0\}$ will be approximated by another process $\{ \mathcal{R}^{(k)}: k \geq 0\}$ whose main characteristic is that its rows $\{ \boldsymbol{\mathcal{R}}_i^{(k)}: k \geq 0, i \in V_n \}$ are conditionally independent of each other given the community labels. We will now describe the approximating process, leaving an intuitive derivation of how this process is obtained to Section~\ref{S.Proofs}. 

To start, define the matrix $M \in [0,1]^{K \times K}$ whose $(r,s)$th component is given by
\begin{equation}\label{average_matrix_M}
     m_{rs} = \dfrac{\pi_s \beta_{r,s}\kappa(s,r)}{\pi_1 \beta_{r,1} \kappa(1,r)+\dots+ \pi_K \beta_{r,K} \kappa(K,r)} \cdot 1\left(\pi_1 \beta_{r,1} \kappa(1,r)+\dots+ \pi_K \beta_{r,K} \kappa(K,r) > 0 \right) ,
\end{equation}
where $\beta_{r,s} = E[ B_{ij} | J_i = r, J_j = s]$. Note that since one may want bot communities to have no inbound neighbors, $M$ may have zero rows. Define also $v_{r,s} = E[B_{ij}^2 | J_i=r, J_j=s]$. Recall that the media signals, given by $\{ \mathbf{Z}_{i}^{(k)}: k \geq 0, i \in V_n\}$, are independent of each other, with $\{ \mathbf{Z}_i^{(k)}: k \geq 0\}$ conditionally i.i.d.~given $\mathbf{A}_i$, with distribution $\nu(\mathbf{A}_i)$. The external signals $\{ \mathbf{W}_{i}^{(k)}: k \geq 0, i \in V_n\}$ follow the form in \eqref{eq:MediaSignals}. We will use $\bar W \in [-1,1]^{K\times \ell}$ to denote the matrix whose $r$th row is $E[ \mathbf{W}_i^{(0)} | J_i = r]$, and  $\bar R \in [-1,1]^{K \times \ell}$ to denote the matrix whose $r$th row is $E[ \mathbf{R}_i^{(0)} | J_i = r]$. 

The approximating process $\{ \mathcal{R}^{(k)}: k \geq 0\}$ is given by: $\mathcal{R}^{(0)} = R^{(0)}$ and
\begin{equation} \label{eq:Limit}
\boldsymbol{\mathcal{R}}_i^{(k)} = \sum_{t=0}^{k-1} (1-c-d)^t \mathbf{W}_i^{(k-t)} + 1( k \geq 2) \sum_{t=1}^{k-1} \sum_{s=1}^t a_{s,t} ( M^s \bar W)_{J_i \bullet} + \sum_{s=1}^k a_{s,k} ( M^s \bar R)_{J_i \bullet} + (1-c-d)^k \mathbf{R}_i^{(0)}, 
\end{equation}
for $k \geq 1$ and $i \in V_n$, where $a_{s,t} = \binom{t}{s} (1-c-d)^{t-s} c^s$ and $H_{j \bullet }$ denotes the $j$th row of matrix $H$ ($H_{\bullet j}$ denotes the $j$th column). Note that conditionally on the vertex attributes $\{ \mathbf{A}_i: i \in V_n\}$, \eqref{eq:Limit} takes the form of an autoregressive process of order one, without any interactions among different vertices. Before stating the main theorems, we will need a few additional definitions. 

For a vector $\mathbf{x}\in\mathbb{R}^n$, we use the standard $l_p$-norm $\|\cdot\|_p$, $p \geq 1$, namely,
$$\| \mathbf{x} \|_p := \left( \sum_{i=1}^n |x_i|^p \right)^{1/p}\quad\text{and}\quad\| \mathbf{x} \|_\infty := \max_{1\leq i\leq n} |x_i|. $$
Applied to a matrix $A\in\mathbb{R}^{n\times\ell}$, $\| \cdot \|_p$ is the induced norm 
$$\| A \|_p := \sup_{\| \mathbf{x} \| \neq 0} \frac{\| A  \mathbf{x} \|_p}{\| \mathbf{x} \|_p}.$$

We will also use the notation 
\[ \mu_r^{(n)} = \sum_{s=1}^K \beta_{r,s} \pi_s^{(n)} \kappa(s,r)  \qquad \text{and} \qquad \nu_r^{(n)} = \sum_{s=1}^K v_{r,s} \pi_s^{(n)} \kappa(s,r) , \]
as well as $$\Delta_n = \max_{r \in \mathcal{I}}\frac{\nu_r^{(n)}}{(\mu_r^{(n)})^2} \qquad\text{and}\qquad \Lambda_n = \max_{r \in \mathcal{I}} \frac{\mu_r^{(n)}}{\nu_r^{(n)}},$$
where $\mathcal{I} = \{ r \in \{1, \dots, K\}: \sum_{s=1}^K m_{rs} > 0 \}$ is the set of non-zero rows of $M$. Note that both $\Delta_n$ and $\Lambda_n$ converge to positive constants as $n \to \infty$ under Assumption~\ref{regularity_condition_communities}. We use $\mathbb{E}_n [ \cdot] = E[ \cdot | \mathscr{J}_n]$ to denote the conditional expectation given the community labels, and $\mathbb{P}_n$ to denote its corresponding conditional probability. The first main result of the paper is given below.

\begin{theorem}\label{MFA_generalthm}
 Define the process $\{\mathcal{R}^{(k)}: k \geq 0\}$ according to \eqref{eq:Limit}. Suppose $\theta_n \geq (6H \Lambda_n)^2 \Delta_n \log n$. Then, there exists a constant $\Gamma < \infty$ such that 
    \begin{equation}\label{convergence_of_infinity_norm}
        \sup\limits_{k\geq0} \, \mathbb{E}_n\left[ \left\| R^{(k)}-\mathcal{R}^{(k)}\right\|_\infty \right] \leq  \Gamma \left( \sqrt{\frac{\log n}{\theta_n}}+ \mathcal{E}_n \right),
    \end{equation}    
where $\mathcal{E}_n := \max_{1\leq r,s \leq K} \left| \frac{\pi_s^{(n)} \pi_r - \pi_s \pi_r^{(n)}}{\pi_r^{(n)} \pi_s} \right|$. 
    Moreover, for any sequence $\theta_n$ satisfying $\theta_n\rightarrow\infty$ as $n\rightarrow\infty$, 
    \begin{equation}\label{convergence_of_1_norm}
       \sup_{k \geq 0} \, \max_{i \in V_n} \mathbb{E}_n\left[ \left\|\mathbf{R}_i^{(k)} - \boldsymbol{\mathcal{R}}_i^{(k)} \right\|_1 \right]  \xrightarrow{P} 0,
    \end{equation}
    as $n\to\infty$. 
\end{theorem}

\begin{remarks} \label{R.Density}
\begin{enumerate}
\item Since $\max_{i \in V_n} \mathbb{E}_n\left[ \left\|\mathbf{R}_i^{(k)} - \boldsymbol{\mathcal{R}}_i^{(k)} \right\|_1 \right] \leq \mathbb{E}_n\left[ \left\| R^{(k)}-\mathcal{R}^{(k)}\right\|_\infty \right]$,  Theorem~\ref{MFA_generalthm} shows that the approximation is stronger when $\theta_n/\log n \to \infty$, and it gradually weakens as the rate at which $\theta_n$ grows drops below the critical rate $\log n$. Intuitively, this can be explained by noting that the average number of neighbors that any vertex has grows with the density parameter $\theta_n$. The larger the number of neighbors, the more their aggregate contributions behave as the average opinion.  The weakest result is valid for any $\theta_n \to \infty$, regardless of how slow the growth is. However, it is worth pointing out that when $\theta_n$ is bounded, the approximation fails, and the correct  approximation is the {\em sparse} approximation, which is determined by the local neighborhood of each vertex (see Theorem~2 in \cite{fraiman2022opinion}).  The threshold $\log n$ determining the strength of the approximation also appears in \cite{Bud_Muk_Wu_19, Bha_Bud_Wu_19} and \cite{avrachenkov2018mean}, where the authors study other processes on graphs with various edge density levels.

\item For each $i \in V_n$, the sequence $\{ \textbf{R}_i^{(k)}: k \geq 0\}$ is called the {\em trajectory} of vertex $i$. Since the rows in the limiting process $\{ \mathcal{R}^{(k)}: k \geq 1\}$ are independent of each other, Theorem~\ref{MFA_generalthm} yields that the trajectories of the process $\{ R^{(k)}: k \geq 0\}$ are asymptotically independent; in other words, the process exhibits {\em propagation of chaos} (note that only particles in the same community are exchangeable). Since
$$  \frac{1}{n} \sum_{i=1}^n \sup_{k \geq 0} \, \mathbb{E}_n\left[ \left\| \mathbf{R}_i^{(k)} - \boldsymbol{\mathcal{R}}^{(k)}_i \right\|_1 \right]  \leq \sup_{k \geq 0} \max_{i \in V_n} \mathbb{E}_n\left[ \left\|\mathbf{R}_i^{(k)} - \boldsymbol{\mathcal{R}}_i^{(k)} \right\|_1 \right] \xrightarrow{P} 0, \qquad n \to \infty,$$
the second statement in Theorem~\ref{MFA_generalthm} implies pointwise propagation of chaos in the sense of Definition~4.1 in \cite{chaintron2022propagationapplications}, throughout the entire range of $k$. We point out that this is not the case in the sparse regime, where the local approximation given by Theorem~2 in \cite{fraiman2022opinion} corresponds to a process whose trajectories are not independent of each other, and the trajectories of vertices close to each other may in fact have significant dependence. 

\item Note that since $\left\| R^{(k)} - \mathcal{R}^{(k)} \right\|_\infty$ is bounded, Theorem~\ref{MFA_generalthm} implies that for $\log n/\theta_n \to 0$, we have that for each $k \geq 0$, 
\[ \left\| R^{(k)} - \mathcal{R}^{(k)} \right\|_\infty \xrightarrow{P} 0, \qquad n \to \infty. \]
\end{enumerate}
\end{remarks}

An alternative way of understanding the approximation is to look at the empirical measure
$$\frac{1}{n} \sum_{i=1}^n 1( \textbf{R}_i^{(0)}, \dots, \textbf{R}_i^{(k)}) \in A),$$
for any measurable set $A \subseteq ([-1,1]^\ell)^{k+1}$. Written in this way, one can identify this measure as the conditional law of a {\em typical opinion}, understood as the trajectory of a uniformly chosen vertex $I_n \in V_n$, given the underlying graph $G_n = (V_n, E_n)$ and the entire process $\{ \mathcal{R}^{(t)}: 0 \leq t \leq k\}$. In our model, the distribution of the typical trajectory converges as $n \to \infty$ to a law that depends only on the community label. Since in the random graph literature the typical vertex converges to the distribution of the root of a rooted graph (usually a branching tree), the notation we use in the following theorem is consistent with that used in the sparse regime, where $\emptyset$ denotes the root of the limiting graph. We point out that in this paper, the limiting graph is not locally finite, so the notion of local weak convergence (see \cite{van2024random2}) does not apply.  


\begin{theorem}\label{thm_LWCprob}
Fix $k \geq 0$ and define the random variables $( \{ \boldsymbol{\mathcal{R}}_\emptyset^{(m)}: 0 \leq m\leq k\}, \mathcal{J}_\emptyset)$ according to:
\begin{align*}
P(\mathcal{J}_\emptyset = s) &= \pi_s, \qquad 1 \leq s \leq K, \\
P\left( \left. (\boldsymbol{\mathcal{R}}_\emptyset^{(0)}, \boldsymbol{\mathcal{R}}_\emptyset^{(1)}, \dots, \boldsymbol{\mathcal{R}}_\emptyset^{(k)}) \in A \right| \mathcal{J}_\emptyset = s \right) &=
\mathbb{P}_n\left(  (\boldsymbol{\mathcal{R}}_i^{(0)}, \boldsymbol{\mathcal{R}}_i^{(1)}, \dots, \boldsymbol{\mathcal{R}}_i^{(k)}) \in A  \right), \qquad A \subseteq ([-1,1]^\ell)^{k+1} , 
\end{align*}
where $\{ \boldsymbol{\mathcal{R}}_i^{(k)}: k \geq 0\}$ is given by \eqref{eq:Limit} and $i$ is any vertex belonging to community $s$. Define $V_{k,i} \in [-1,1]^{\ell \times (k+1)}$ to be the matrix having $j$th column $(V_{k,i})_{\bullet j} = (\mathbf{R}_i^{(j-1)})^\top$, and let $\mathcal{V}_k \in [-1,1]^{\ell \times (k+1)}$ be the matrix having $j$th column $(\mathcal{V}_k)_{\bullet j} = (\boldsymbol{\mathcal{R}}_\emptyset^{(j-1)})^\top$. Then, 
for any $1\leq r \leq K$ and any bounded and continuous (with respect to the $\| \cdot \|_1$ operator norm) function $f: [-1,1]^{\ell \times (k+1)} \to \mathbb{R}$, we have 
    \begin{equation*}
        \dfrac{1}{n}\sum_{i=1}^n f\left( V_{k,i} \right) 1(J_i = r) \stackrel{P}{\rightarrow} E\left[f\left( \mathcal{V}_k \right) 1(\mathcal{J}_\emptyset = r) \right]
    \end{equation*}
    as $n\rightarrow\infty$.  Furthermore,  for any arbitrary collection of vertices $\{i_1, \dots, i_m\} \subseteq V_n$ having community labels $\{r_1, \dots, r_m\}$, $m \geq 1$, and any set of continuous (with respect to $\| \cdot \|_1$) bounded functions $\{ f_1, \dots, f_m\}$ on $[-1,1]^{\ell \times (k+1)}$, $k \geq 0$, we have
     \[ \mathbb{E}_n\left[ \prod\limits_{j=1}^m f_j\left( V_{k,i_j} \right) \right] \xrightarrow{P} \prod\limits_{j=1}^m E\left[ \left. f_j \left( \mathcal{V}_k \right) \right| \mathcal{J}_\emptyset = r_j \right], \qquad n \to \infty. \]       
\end{theorem}

\begin{remark}
The last statement in Theorem~\ref{thm_LWCprob} is closely related to Theorem~2.6.B in \cite{fraiman2022stochastic}, which states that trajectories of a finite set of vertices uniformly chosen at random are asymptotically independent of each other, a limited form of propagation of chaos that holds only for vertices that are likely to be far from each other.  However, the last statement in Theorem~\ref{thm_LWCprob} allows the vertices to be chosen arbitrarily. 
\end{remark}

The last set of results in the paper refers to the stationary behavior of the Markov chain $\{ R^{(k)}: k \geq 0\}$. Recall that $R$ denotes a matrix in $[-1,1]^{n\times\ell}$ having the stationary distribution of the chain, and $\mathbf{R}_{I_n}$ denotes a uniformly chosen row of the matrix $R$. Note that $\textbf{R}_{I_n}$ represents the {\em typical stationary opinion} on the graph, across all $\ell$ topics. Our last theorem provides a characterization for $\textbf{R}_{I_n}$ when $n$ is large, by establishing the existence of a random variable $\boldsymbol{\mathcal{R}}_\emptyset\in[-1,1]^{\ell}$ such that
$$\textbf{R}_{I_n} \Rightarrow \boldsymbol{\mathcal{R}}_\emptyset, \qquad n \to \infty,$$
where $\Rightarrow$ denotes weak convergence on $[-1,1]^\ell$. Since $\boldsymbol{\mathcal{R}}_\emptyset$ also turns out to be the weak limit of $\{ \boldsymbol{\mathcal{R}}_\emptyset^{(k)}: k \geq 0\}$ as $k \to \infty$, our results show that we can exchange the limits in $k$ (time) and $n$ (number of vertices). 

\begin{theorem}\label{thm_convergence_in_prob}
 Define the random variables $(\boldsymbol{\mathcal{R}}_\emptyset, \mathcal{J}_\emptyset)$ according to: 
 \begin{align*}
P( \boldsymbol{\mathcal{R}}_\emptyset \in A | \mathcal{J}_\emptyset = r) &= \mathbb{P}_n \left( \left. \sum_{t=0}^\infty (1-c-d)^t \mathbf{W}_i^{(t)} + \sum_{t=1}^\infty \sum_{s=1}^t a_{s,t} (M^s \bar W)_{J_i}  \in A \right| J_i = r \right), \qquad A \subseteq [-1,1]^{\ell}, \\
P(\mathcal{J}_\emptyset = r) &= \pi_r, \qquad 1 \leq r \leq K.
\end{align*} 
 Then, there exists a coupling $(\textbf{R}_{I_n}, J_{I_n}, \boldsymbol{\mathcal{R}}_\emptyset, \mathcal{J}_\emptyset)$, such that for any continuous (with respect to $\| \cdot \|_1$) and bounded function $f: [-1, 1]^{\ell}\rightarrow\mathbb{R}$, we have for any $1 \leq r \leq K$,
    \[ \mathbb{E}_n\left[ \| \textbf{R}_{I_n}-\boldsymbol{\mathcal{R}}_\emptyset \|_1 \right]\stackrel{P}{\rightarrow}0 \quad\text{and}\quad \dfrac{1}{n}\sum_{i=1}^n f(\textbf{R}_i) 1(J_i = r) \stackrel{P}{\rightarrow} E[f(\boldsymbol{\mathcal{R}}_\emptyset) 1(\mathcal{J}_\emptyset = r) ] \]
    as $n\rightarrow\infty$. Furthermore, for any arbitrary collection of vertices $\{i_1, \dots, i_m\} \subseteq V_n$ having community labels $\{r_1, \dots, r_m\}$, $m \geq 1$, and any set of continuous (with respect to $\| \cdot \|_1$) bounded functions $\{ f_1, \dots, f_m\}$ on $[-1, 1]^{\ell}$, we have
 \[   \mathbb{E}_n\left[ \prod_{j=1}^m f_j(\textbf{R}_{i_j}) \right] \xrightarrow{P} \prod_{j=1}^m E[ f_j(\boldsymbol{\mathcal{R}}_\emptyset) | \mathcal{J}_\emptyset = r_j], \qquad n \to \infty. \]
  \end{theorem}

\begin{remark}
    Note that the shape of $\boldsymbol{\mathcal{R}}_\emptyset$ is simply the limit as $k \to \infty$ of $\boldsymbol{\mathcal{R}}_{I_n}^{(k)}$ after a time reversal of the external signals, i.e., $\mathbf{W}_{I_n}^{(k-t)} \stackrel{\mathcal{D}}{=} \mathbf{W}_{I_n}^{(t)}$. 
\end{remark}

The proofs of all the theorems are contained in Section~\ref{S.Proofs}.

\section{Related Literature}\label{lit-review}

Interest in the evolution of opinions on social networks dates back to the 50's, with the work of French \cite{French56} and DeGroot  \cite{degroot1974reaching}. The model in \cite{degroot1974reaching} corresponds in our notation to the recursion $\mathbf{R}^{(k+1)} = A \mathbf{R}^{(k)}$, for a given $\mathbf{R}^{(0)}$ and $A = c \, C + (1-c) I$, with $C = (c_{ij})$ as in \eqref{eq:CommunityMatrix} (the model is for a single topic, hence the vectors instead of matrices). The main question studied in \cite{degroot1974reaching, berger1981necessary} was the existence of consensus, seen as a limit $\mathbf{R} = \lim_{k \to \infty} \mathbf{R}^{(k)}$ whose components are all equal to each other (a consequence of the Perron-Frobenius theorem when the matrix $A$ is irreducible and aperiodic). In \cite{friedkin1990social}, Friedkin and Johnsen extended the model to allow disagreement by considering the recursion $\mathbf{R}^{(k+1)} = A \mathbf{R}^{(k)} + \mathbf{Q}$, with $\mathbf{Q}$ a deterministic vector, which is a special case of our model.  The models in \cite{Yang2017InnovationGE} replace the deterministic $\mathbf{Q}$ with a sequence of random signals $\{ \mathbf{W}^{(k)}: k \geq 0\}$, as we do, although they are independent of the underlying network, since they are used to model miscommunication. A more recent generalization of the DeGroot model was given in \cite{dandekar2013biased}, where the authors introduced the concept of \textit{biased assimilation}, namely the tendency of individuals to pull their opinion towards their initial opinion once presented with inconclusive evidence, in order to make the repeated averaging scheme of DeGroot lead to polarization. Another closely related model that is capable of introducing confirmation bias by rejecting an opinion from a neighbor when it is too different from their own expressed opinion is the Hegselmann-Krause model \cite{Rainer2002-RAIODA}, which due to its non-linear updating rule is not covered by our analysis. All these models are defined on a fixed graph, so the stochastic matrix $A$ is deterministic, unlike ours, where we use both the updating rules of the opinion process and the properties of the underlying graph to explain complex phenomena.  Other variants of the DeGroot and Friedkin-Johnsen models that incorporate multiple topics and time-varying self-weights (e.g., $d$ in our model is allowed to be different for each vertex and change with time) are given in \cite{jia2015opinion, parsegov2016novel, tian2021social}.  

While all the models mentioned above are defined on discrete time with synchronous updating rules, there is a different family of popular opinion models where interactions occur one at a time, at either random or deterministic times. Among those are the bounded confidence models, such as the Deffuant-Weisbuch model \cite{weisbuch1999dynamical, deffuant2000mixing, deffuant2002can}. In this model, at each time step a pair of neighboring individuals is randomly selected. Then, the two neighbors update their opinions according to a simple linear equation if, and only if, the difference of their current opinions is less than a specified threshold. This model, like the Hegselmann-Krause model, is meant to explain fragmentation and polarization through a direct confirmation bias. Its non-linearity and lack of an external signal means that its analysis is usually done using the methods of interacting particle systems or numerical stochastic simulation \cite{holyst2001social, ben2003bifurcations, le2007generic, 8c5aa4c6-94cc-393f-a938-57db54441872, gomez2012bounded, meng2018opinion}. A more recent analysis of the Deffuant-Weisbuch model that takes into account the network's topology is \cite{fennell2021generalized}. Our model is not meant to  tackle confirmation bias at the level of the opinions being shared, but rather at the level of the individuals who share them, and the main appeal of our model is that it accounts for the combined effects of the network topology, the media framing effect, and people's selective exposure, while still producing computable formulas for means and variances for the opinion process. 

As mentioned in the introduction, the main focus of this paper is not the explanatory and predictive power of our model, but rather its behavior on graphs of various levels of edge density. So, before we discuss that aspect of the related literature, we just want to mention that there is yet another family of classical opinion models taking values on finite sets (rather than continuous intervals), like the voter model \cite{liggett1999stochastic}, where opinions can be either 0 or 1, or the Axelrod model  \cite{axelrod1997dissemination}, where each vertex has a fixed set of features and traits, leading to a notion of ``similarity" between vertices. The former can be used to model consensus through asynchronous updates, whereas the latter achieves a diversity of opinions by using a trait updating rule that depends on how similar the features of the two interacting vertices are, a concept known as {\em homophily}. The underlying graph in both of these models is usually a lattice, although the voter model has also been analyzed in heterogeneous graphs in \cite{sood2005voter}. The work in \cite{redner2019reality} contains a review of interesting variants of the voter model. 

Our opinion model falls in the category of interacting particle systems in discrete time with synchronous updates and continuous state space, although only vertices in the same community are exchangeable. The approximating process $\{ \mathcal{R}^{(k)}: k \geq 0\}$ is a classical mean-field approximation, where the complex interactions between neighboring vertices are replaced with their average behavior. Our main theorem proves propagation of chaos by coupling the trajectories, as described in the surveys \cite{chaintron2022propagationapplications,chaintron2021propagation}. The literature on mean-field approximations is so vast, and covers such a wide range of fields, that it would be impossible to mention all the related models for which results similar to ours have been derived. The closest work to the techniques used in this paper is the analysis of the personalized PageRank algorithm done in \cite{avrachenkov2018mean}, where a stochastic block model is the only source of randomness in what otherwise would be a deterministic linear recursion. The main results in \cite{avrachenkov2018mean} are closely related to our Theorem~\ref{MFA_generalthm} for density parameters $\theta_n/\log n \to \infty$, albeit for a different recursion and under an $l_2$-norm (rather than the $l_\infty$-norm we use). Our work shows that the approximation itself continues to hold beyond the $\log n$ threshold, for any $\theta_n \to \infty$, but with lesser precision (a weaker norm). Relative to the majority of mean-field results in the interacting particles literature, our work is different in the explicit role that the underlying random graph plays, its ability to seamlessly cover the entire range of edge densities (other than the sparse case), the explicit characterization of the limiting process as the number of particles goes to infinity, and the exchange of limits with the stationary version of the limiting process. Other works that analyze stochastic processes on graphs with various levels of edge density are \cite{Bud_Muk_Wu_19}, which studies the supermarket model on graphs, and \cite{Bha_Bud_Wu_19}, which analyzes weakly interacting diffusions on time varying random graphs.  It is also worth pointing out that mean-field approximations for bounded confidence models such as the Deffuant-Weisbuch model and some of its variants have been given in \cite{le2007generic, 8c5aa4c6-94cc-393f-a938-57db54441872,gomez2012bounded}.  

Finally, we would like to point out that the analysis of our model in the semi-sparse regime, i.e., $\theta_n$ below the $\log n$ threshold, is based on a coupling of the underlying random graph with a marked multi-type branching process. This technique is based on the notion of local weak convergence of random graph sequences \cite{10.1214/EJP.v6-96, aldous2004objective} (see also \cite{van2024random2}). Although for non-sparse random graphs the local weak limit is not a locally bounded graph, the dSBM can still be coupled with a marked multi-type Galton-Watson tree (with average degree that grows with $n$). In fact, a strong coupling, in the sense of the intermediate coupling in \cite{olvera2022strong}, still holds, which is what enables the work in this paper. For sparse graphs, local weak convergence or strong couplings have been used to study the PageRank algorithm in \cite{chen2014pagerank, garavaglia2020local, banerjee2022pagerank}, general discrete-time stochastic recursions on sparse graphs in \cite{fraiman2022stochastic}, interacting diffusions in \cite{lacker2019local, lacker2020marginal}, and continuous jump process models \cite{ganguly2022hydrodynamic, ganguly2022interacting, cocomello2023exact}. The techniques in this paper will likely extend to more general classes of stochastic recursions on random graphs with growing edge densities.

\section{Proofs}\label{S.Proofs}

This section contains the proofs for all the theorems in Section~\ref{main_results}. The proof of our main result, Theorem~\ref{MFA_generalthm}, requires two different approaches depending on the speed at which $\theta_n \to \infty$, however, both approaches start with the same intermediate approximation. We will start by presenting the intermediate approximation in this section, and then complete the analysis for the two different density regimes; in Section~\ref{SS.SemiDense} for the case when $\theta_n \geq \gamma \log n$ for some $\gamma > 0$, and in Section~\ref{SS.SemiSparse} for the case when $\limsup_{n\to\infty}\log n/\theta_n>0$.
Finally, as mentioned earlier, Theorems~\ref{thm_LWCprob} and \ref{thm_convergence_in_prob} are consequences of Theorem~\ref{MFA_generalthm}, so we give their proofs in Section~\ref{SS.LWCprob}. 

Recall that the opinion process will be constructed on a fixed directed graph $G_n$ sampled according to Section~\ref{subsection_network}, where each vertex $i \in V_n$ has an attribute $\mathbf{A}_i = (J_i, \mathbf{Q}_i)$ and an extended attribute $\mathbf{Y}_i = (J_i, \mathbf{Q}_i, \mathbf{B}_i)$, with $\mathbf{B}_i = (B_{ij})_{1 \leq j \leq n}$ the unnormalized weights. 

In order for us to derive the intermediate approximation, we first note that the opinion recursion (\ref{eqn:recursion_of_opinions_autoregressive}) can be rewritten in matrix notation as 
\begin{equation}\label{opinion_recursion_matrix_form}
    R^{(k)}=AR^{(k-1)}+W^{(k)},
\end{equation}
where $R^{(k)} \in [-1,1]^{n \times \ell}$ and $W^{(k)} \in [-1,1]^{n \times \ell}$ are the matrices having $(i,j)$th components $R_{ij}^{(k)}$ and $W_{ij}^{(k)}$, respectively, and $A \in [0,1]^{n\times n}$ the matrix satisfying
$$A_{ij} = c\cdot C_{ij}, \quad i \neq j, \qquad A_{ii} = 1-c-d,$$
with $C_{ij}$ constructed according to \eqref{eq:CommunityMatrix} using the random unnormalized weighs $\{B_{ij}: i,j \in V_n\}$ (the use of upper case is to remind the reader that these are random). 
Iterating \eqref{opinion_recursion_matrix_form} gives the explicit solution
$$R^{(k)} = \sum_{t=0}^{k-1} A^t W^{(k-t)} + A^k R^{(0)}.$$
Using the observation that $A = c\,C + (1-c-d) I$, where $C$ is the matrix having $(i,j)$th component $C_{ij}$, and $I$ is the identity matrix in $\mathbb{R}^{n\times n}$, we get
$$A^t = (c\,C + (1-c-d)I)^t = \sum_{s=0}^t \binom{t}{s} (1-c-d)^{t-s}c^s C^s = \sum_{s=0}^t a_{s,t} C^s, \qquad t\geq 0,$$
and, therefore,
\begin{equation} \label{eq:SolvedRecursion}
R^{(k)} = \sum_{t=0}^{k-1} \sum_{s=0}^t a_{s,t} C^s W^{(k-t)} + \sum_{s=0}^k a_{s,k} C^s R^{(0)}, \qquad k \geq 1.
\end{equation}

Now, construct the matrix $\tilde M \in \mathbb{R}^{n\times n}$ according to
$$\tilde M_{ij} =  \frac{\beta_{J_i,J_j} \kappa(J_j, J_i)}{n \left( \beta_{J_i,1} \pi_1^{(n)}\kappa(1,J_i)+\dots+\beta_{J_i,K} \pi_K^{(n)}\kappa(K,J_i) \right)} 1(i\neq j),$$
and recall that $\beta_{r,s} = E[ B_{ij} | J_i = r, J_j = s]$, with $\tilde M_{ij} \equiv 0$ if the denominator is zero.  
The intuition behind this definition of $\tilde{M}$ is that it is the ``approximate" conditional mean of $C$ given the community labels $\mathscr{J}_n$, meaning that for $i, j\in V_n$, $i\neq j$, and $n$ sufficiently large (so that the minimum in $p_{ij}^{(n)}$ can be dropped), we have
\[ \tilde{M}_{ij} = \dfrac{\beta_{J_i,J_j} p_{ij}^{(n)} }{n \left( \pi_1^{(n)}\beta_{J_i,1} p_{1i}^{(n)} +\dots+\pi_K^{(n)}\beta_{J_i,K}p_{Ki}^{(n)} \right) } =
\dfrac{\mathbb{E}_n[B_{ij}1(j\to i)]}{\mathbb{E}_n\left[\, \sum_{r=1}^n B_{ir}1(r\to i) \right]}.
\]
Define $\breve{W} = \mathbb{E}_n\left[ W^{(0)} \right]$, $\breve{R} = \mathbb{E}_n\left[ R^{(0)} \right]$, and construct the new process $\{\tilde{R}^{(k)}: k\geq0\}$ defined by
$$ \tilde{R}^{(0)} = R^{(0)},$$
\begin{equation}\label{interm_process}
    \tilde{R}^{(k)} =  \sum_{t=0}^{k-1} (1-c-d)^t W^{(k-t)} + 1(k \geq 2) \sum_{t=1}^{k-1}  \sum_{s=1}^t a_{s,t} \tilde M^s \breve{W}  + \sum_{s=1}^k a_{s,k} \tilde M^s \breve{R} + (1-c-d)^k R^{(0)}, \quad k \geq 1.
\end{equation}

\noindent
Intuitively, $\tilde{R}^{(k)}$ replaces all neighbor contributions with their approximate means, i.e., every term of the form $C^s X$ with $s \geq 1$ and $X$ a random matrix is replaced with $\tilde M^s \mathbb{E}_n[X]$. It is important to point out that the rows of $\tilde{R}^{(k)}$, which we will denote $\{ \mathbf{\tilde R}_i^{(k)}: i \in V_n\}$,  become independent as a consequence. Our first result computes a bound for $\mathbb{E}_n\left[ \left\| R^{(k)} - \tilde{R}^{(k)} \right\|_\infty \right]$ and $\max_{i \in V_n} \mathbb{E}_n\left[ \left\| \mathbf{R}_i^{(k)} - \mathbf{\tilde R}_i^{(k)} \right\|_i \right]$ .

\begin{theorem}\label{T.interm_pnorm} 
For any $k \geq 1$, we have
\begin{align*}
\max_{i \in V_n} \mathbb{E}_n\left[ \left\| \mathbf{R}_i^{(k)} - \mathbf{\tilde R}_i^{(k)} \right\|_1 \right] &\leq  1(k \geq 2)  \sum_{t=1}^{k-1} \sum_{s=1}^t a_{s,t}  \max_{i\in V_n} \mathbb{E}_n\left[ \left\| (C^s W^{(0)})_{i\bullet} - (\tilde M^s \breve{W})_{i\bullet} \right\|_1 \right] \\
&\hspace{5mm} + \sum_{s=1}^k a_{s,k} \max_{i\in V_n} \mathbb{E}_n\left[  \left\| (C^s R^{(0)})_{i\bullet} - (\tilde M^s \breve{R})_{i\bullet} \right\|_1 \right]. 
\end{align*}
Furthermore, for $h(X) = \mathbb{E}_n\left[ \left\| (C-\tilde M) X \right\|_\infty \right]$, we have  
    \begin{align*}
  \mathbb{E}_n\left[ \left\| R^{(k)} - \tilde{R}^{(k)} \right\|_\infty \right]  &\leq  \max_{r\geq 1}  h( \tilde{M}^r \breve{W}  )  \sum_{t=0}^{k-1} (1-d)^{t-1} t + \mathbb{E}_n \left[  \left\| C W^{(0)} - \tilde M \breve{W} \right\|_\infty \right] \sum_{t=0}^{k-1} (1-d)^t  \\
&\hspace{5mm} + \max_{r\geq 1}  h( \tilde{M}^r \breve{R} ) \, k (1-d)^{k-1}   + \mathbb{E}_n \left[  \left\| C R^{(0)} - \tilde M \breve{R} \right\|_\infty \right] (1-d)^k  .
\end{align*}
\end{theorem}

\begin{proof}
Recall that $a_{s,t} = \binom{t}{s} (1-c-d)^{t-s}c^s$ and note that for any $i \in V_n$ we have
\begin{align*}
\left\| \mathbf{R}_i^{(k)} - \mathbf{\tilde R}^{(k)}_i \right\|_1 &= \left\| 1(k\geq 2) \sum_{t=1}^{k-1} \sum_{s=1}^t a_{s,t} \left( C^s W^{(k-t)} - \tilde M^s \breve W \right)_{i\bullet} + \sum_{s=1}^k a_{s,k} \left( C^s R^{(0)} - \tilde M^s \breve R \right)_{i\bullet}  \right\|_1 \\
&\leq 1(k\geq 2) \sum_{t=1}^{k-1} \sum_{s=1}^t a_{s,t}  \left\| \left( C^s W^{(k-t)} - \tilde M^s \breve W \right)_{i\bullet} \right\|_1 + \sum_{s=1}^k a_{s,k} \left\| \left( C^s R^{(0)} - \tilde M^s \breve R \right)_{i\bullet}  \right\|_1.
\end{align*}
To obtain the first statement of the theorem simply take expectations on both sides to get
\begin{align*}
 \max_{i \in V_n}   \mathbb{E}_n\left[ \left\| \mathbf{R}_i^{(k)} - \mathbf{\tilde R}^{(k)}_i \right\|_1 \right] &\leq  1(k\geq 2) \sum_{t=1}^{k-1} \sum_{s=1}^t a_{s,t}  \max_{i \in V_n} \mathbb{E}_n\left[ \left\| \left( C^s W^{(k-t)} - \tilde M^s \breve W \right)_{i\bullet} \right\|_1 \right] \\
    &\hspace{5mm} + \sum_{s=1}^k a_{s,k}  \max_{i \in V_n} \mathbb{E}_n\left[  \left\| \left( C^s R^{(0)} - \tilde M^s \breve R \right)_{i\bullet}  \right\|_1 \right].
\end{align*}

To continue to the second statement, note that from the properties of the norm $\| \cdot \|_\infty$, we obtain that
\begin{align*}
\left\| R^{(k)} - \tilde R^{(k)} \right\|_\infty &\leq 1(k\geq 2) \sum_{t=1}^{k-1} \sum_{s=1}^t a_{s,t}  \left\| C^s W^{(k-t)} - \tilde M^s \breve W  \right\|_\infty + \sum_{s=1}^k a_{s,k} \left\|  C^s R^{(0)} - \tilde M^s \breve R   \right\|_\infty.
\end{align*}
Furthermore, since for any matrices $X, \breve X \in \mathbb{R}^{n \times \ell}$ and any $s \geq 1$ we have
\begin{align*}
    C^s X - \tilde M^s \breve X &= C^{s-1}(C X - \tilde M \breve X) + (C^{s-1} - \tilde M^{s-1}) \tilde M \breve X \\ &=
    C^{s-1} (C X - \tilde M \breve X) + \sum_{r=1}^{s-1} C^{s-r-1} (C - \tilde M) \tilde M^r \breve X,
\end{align*}
then
\begin{align*}
 \left\| C^s X - \tilde M^s \breve X \right\|_\infty &\leq \sum_{r=1}^{s-1} \left\| C^{s-r-1} (C-\tilde M) \tilde M^r \breve X  \right\|_\infty + \left\| C^{s-1} (C X - \tilde M \breve X) \right\|_\infty   \\
&\leq \sum_{r=1}^{s-1} \| C \|_\infty^{s-r-1} \left\| (C-\tilde M) \tilde M^r \breve X \right\|_\infty + \| C\|_\infty^{s-1} \left\| CX - \tilde M \breve X \right\|_\infty \\
&= \sum_{r=1}^{s-1}  \left\| (C-\tilde M) \tilde M^r \breve X \right\|_\infty +  \left\| CX - \tilde M \breve X \right\|_\infty,
\end{align*}
where we used the fact that $\| C \|_\infty = 1$. We conclude that
\begin{align*}
&\mathbb{E}_n\left[ \left\| R^{(k)} - \tilde R^{(k)} \right\|_\infty \right] \\
&\leq 1(k \geq 2) \sum_{t=1}^{k-1} \sum_{s=1}^t a_{s,t} \mathbb{E}_n\left[  \left\| C^s W^{(k-t)} - \tilde M^s \breve{W} \right\|_\infty  \right]+ \sum_{s=1}^k a_{s,k} \mathbb{E}_n\left[  \left\| C^s R^{(0)} - \tilde M^s \breve{R} \right\|_\infty \right] \\
&\leq 1(k \geq 2)  \sum_{t=1}^{k-1} \sum_{s=1}^t a_{s,t} \left( \sum_{r=1}^{s-1} \mathbb{E}_n\left[  \left\| (C- \tilde M) \tilde M^r \breve W \right\|_\infty \right]   + \mathbb{E}_n \left[  \left\| C W^{(k-t)}  - \tilde M \breve W \right\|_\infty \right]    \right)  \\
&\hspace{5mm} + \sum_{s=1}^k a_{s,k}  \left( \sum_{r=1}^{s-1} \mathbb{E}_n\left[  \left\| (C- \tilde M) \tilde M^r \breve{R}  \right\|_\infty \right]   + \mathbb{E}_n \left[  \left\| C R^{(0)} - \tilde M \breve{R} \right\|_\infty \right]    \right)   \\
&\leq  \max_{r\geq 1}  h( \tilde{M}^r \breve{W}  ) 1(k \geq 2)  \sum_{t=1}^{k-1} \sum_{s=1}^t a_{s,t} (s-1) + \mathbb{E}_n \left[  \left\| C W^{(0)} - \tilde M \breve{W}  \right\|_\infty \right]  1(k \geq 2)  \sum_{t=1}^{k-1} \sum_{s=1}^t a_{s,t}  \\
&\hspace{5mm} + \max_{r\geq 1}  h( \tilde{M}^r \breve{R} )  \sum_{s=1}^k a_{s,k}(s-1)   + \mathbb{E}_n \left[  \left\| C R^{(0)} - \tilde M \breve{R}  \right\|_\infty \right]  \sum_{s=1}^k a_{s,k}.
\end{align*}
To compute the double sums more explicitly, let $\text{Bin}(n,p)$ denote a binomial random variable with parameters $(n,p)$, and note that for $k \geq 1$,
\begin{align} \label{eq:BinMean}
 \sum_{t=1}^{k} \sum_{s=1}^t a_{s,t}  s &=  \sum_{t=1}^{k} \sum_{s=1}^t \binom{t}{s} (1-c-d)^{t-s}  c^{s} s = \sum_{t=1}^k (1-d)^t  E[ \text{Bin}(t, c/(1-d)) ]= c \sum_{t=1}^k (1-d)^{t-1} t,
\end{align}
and
\begin{align} \label{eq:BinProb}
 \sum_{t=1}^{k} \sum_{s=1}^t a_{s,t}   &=  \sum_{t=1}^{k} \sum_{s=1}^t \binom{t}{s} (1-c-d)^{t-s}  c^{s}  = \sum_{t=1}^k (1-d)^t  P( \text{Bin}(t, c/(1-d)) \geq 1) \nonumber \\
 &= \sum_{t=1}^k (1-d)^t  \left(1 - (1-c-d)^t /(1-d)^t \right) =  \sum_{t=1}^k \left( (1-d)^t - (1-c-d)^t \right). 
 \end{align}
 Similarly, the single sums can be computed to be
 \begin{equation}\label{eq:BinMean2}
     \sum_{s=1}^k a_{s,k}  s = \sum_{s=1}^k \binom{k}{s} (1-c-d)^{k-s} c^{s} s = (1-d)^k  E[\text{Bin}(k,c/(1-d))] = c k (1-d)^{k-1}
 \end{equation}
 and
\begin{equation}\label{eq:BinProb2}
    \sum_{s=1}^k a_{s,k}  = \sum_{s=1}^k \binom{k}{s} (1-c-d)^{k-s} c^{s} = (1-d)^k P(\text{Bin}(k,c/(1-d)) \geq 1) =  (1-d)^k  - (1 - c - d)^k.
\end{equation}
Consequently,
\begin{align*}
\mathbb{E}_n\left[ \left\| R^{(k)} - \tilde R^{(k)} \right\|_\infty \right] 
&\leq  \max_{r\geq 1}  h( \tilde{M}^r \breve{W}  ) 1(k \geq 2)  \sum_{t=1}^{k-1} (1-d)^{t-1} t \\
&\hspace{5mm} +\mathbb{E}_n \left[  \left\| C W^{(0)} - \tilde M \breve{W}  \right\|_\infty \right]  1(k \geq 2) \sum_{t=1}^{k-1} (1-d)^t + \\ 
&\hspace{5mm} +\max_{r\geq 1}  h( \tilde{M}^r \breve{R} ) \, k (1-d)^{k-1}   + \mathbb{E}_n \left[  \left\| C R^{(0)} - \tilde M \breve{R} \right\|_\infty \right] (1-d)^k.
\end{align*}
\end{proof}

Once the intermediate approximation $\{ \tilde{R}^{(k)}: k \geq 0\}$ given by \eqref{interm_process} is obtained, its connection to the final approximation $\{ \mathcal{R}^{(k)}: k \geq 0\}$  given by \eqref{eq:Limit} becomes apparent, since for all $s \geq 1$ we have 
$$(\tilde M^s \breve{W})_{i \bullet} \approx (M^s \bar{W})_{J_i \bullet} \qquad \text{and} \qquad (\tilde M^s \breve{R})_{i \bullet} \approx  (M^s \bar{R})_{J_i \bullet} ,$$
for large $n$. The following result provides a bound for the norm between these two processes.

\begin{theorem}\label{T.interm_meanfield}
For any $\theta_n \to \infty$, 
\[ \sup_{k \geq 0} \left\| \tilde{R}^{(k)} - \mathcal{R}^{(k)} \right\|_\infty \leq \frac{\ell c}{d^2} \cdot \mathcal{E}_n, \]
where $\mathcal{E}_n = \max_{1\leq r,s \leq K} \left| \frac{\pi_s^{(n)} \pi_r - \pi_s \pi_r^{(n)}}{\pi_r^{(n)} \pi_s} \right|$. 
\end{theorem}

\begin{proof}
Note that for any $1 \leq j \leq K$, 
\begin{align*}
&\left\| \tilde R^{(k)}  - \mathcal{R}^{(k)}  \right\|_\infty  \\
&=  \max_{1\leq i\leq n} \sum_{j=1}^\ell \left| 1(k\geq2) \sum_{t=1}^{k-1}\sum_{s=1}^t a_{s,t} (\tilde{M}^s \breve{W})_{ij}  +\sum_{s=1}^k a_{s,k} (\tilde{M}^s \breve{R})_{ij}  - \right. \\ &\hspace{40mm}
\left. -1(k\geq2) \sum_{t=1}^{k-1}\sum_{s=1}^t a_{s,t} (M^s \bar{W})_{J_i j}-\sum_{s=1}^ka_{s,k} (M\bar{R})_{J_i j} \right| \\ &=
\max_{1 \leq i \leq n} \sum_{j=1}^\ell \left| 1(k\geq 2) \sum_{t=1}^{k-1} \sum_{s=1}^t a_{s,t} \left( (\tilde M^s \breve{W})_{ij} -  (M^s \bar{W})_{J_i j}  \right)  + \sum_{s=1}^k a_{s,k} \left( (\tilde M^s \breve{R})_{ij} -  (M^s \bar{R})_{J_i j} \right) \right| \\ &\leq
1(k\geq 2) \sum_{t=1}^{k-1} \sum_{s=1}^t a_{s,t} \max_{1\leq i \leq n} \sum_{j=1}^\ell \left| (\tilde M^s \breve{W})_{ij} -  (M^s \bar{W} )_{J_i j}   \right| +  \sum_{s=1}^k a_{s,k} \max_{1 \leq i \leq n} \sum_{j=1}^\ell \left| (\tilde M^s \breve{R} )_{ij} -  (M^s \bar{R} )_{J_i j} \right| .
\end{align*}

Note that $\tilde M \in \mathbb{R}^{n\times n}$ is such that all the rows $\mathcal{S}_r = \{ \tilde M_{i \bullet} : J_i = r \}$ are equal to each other, all the columns $\{ \tilde M_{\bullet j}: J_j = r \}$ are equal to each other, and the cardinality of the $\mathcal{S}_r$ is $n \pi_r^{(n)}$, for $1 \leq r \leq K$. Similarly, the vector $\breve{W}_{\bullet j} \in \mathbb{R}^n$ satisfies that all the components $\{ \breve{W}_{i j}: J_i = r\}$ are equal to each other and the order in which the vertex labels appear is the same as the columns of $\tilde M$. Define the matrix $\breve M \in \mathbb{R}^{K \times K}$ according to
\begin{equation} \label{eq:BreveM}
\breve M_{rs} =  \frac{\beta_{r, s}  \pi_s^{(n)} \kappa(s, r)}{ \beta_{r, 1} \pi_1^{(n)} \kappa(1,r) + \dots + \beta_{r, K} \pi_K^{(n)} \kappa(K, r)}, \qquad 1 \leq r,s \leq K,
\end{equation}
with $\breve M_{rs} \equiv 0$ if the denominator is zero, and note that
$$\breve M_{rs} = \sum_{j=1}^n \tilde M_{ij} 1(J_j = s)  \qquad \text{for any $i$ such that $J_i = r$}.$$
It follows that for any $i \in V_n$,
\begin{equation}\label{Mtilde_Mbreve}
    (\tilde M \breve{W} )_{ij}  = \sum_{l=1}^n \tilde M_{il} \breve{W}_{lj} = \sum_{r=1}^K \sum_{l=1}^n  \tilde M_{il} \breve{W}_{l j} 1(J_l = r) = \sum_{r=1}^K \breve M_{J_i,r} \bar{W}_{rj} = (\breve M \bar{W})_{J_i j}.
\end{equation}
Now suppose that $(\tilde M^{s-1} \breve{W})_{ij} = (\breve{M}^{s-1} \bar{W})_{J_i j}$ for $s \geq 2$, and note that
\begin{align*}
(\tilde M^s \breve{W})_{ij} &= \sum_{l=1}^n \tilde M_{il} (\tilde M^{s-1} \breve{W})_{lj} = \sum_{r=1}^K  \sum_{l=1}^n \tilde M_{il} (\breve{M}^{s-1} \bar{W})_{J_l j} 1(J_l = r) =\sum_{r=1}^K \breve{M}_{J_i r} (\breve{M}^{s-1} \bar{W})_{r j} = (\breve{M}^s \bar W)_{J_i j},
\end{align*}
therefore, by induction, we conclude that
$$(\tilde M^s \breve{W})_{ij} = (\breve M^s \bar{W})_{J_i j} \qquad \text{for all } s \geq 1 \text{ and all } i \in V_n.$$
The same arguments also yield $(\tilde M^s \breve{R})_{ij} = (\breve M^s \bar{R})_{J_i j}$ for all $s \geq 1$ and all $i \in V_n$. 
Hence,
\[
\max_{1\leq i \leq n} \sum_{j=1}^\ell \left| (\tilde M^s \breve{W})_{ij} -  (M^s \bar{W})_{J_i j}   \right| = \max_{1\leq i \leq n} \sum_{j=1}^\ell \left| (\breve M^s \bar{W})_{J_i j} - (M^s \bar{W})_{J_i j} \right| = \| ( \breve M^s - M^s) \bar W \|_\infty \]
and
\[
 \max_{1\leq i \leq n} \sum_{j=1}^\ell \left| (\tilde M^s \bar{R})_{ij} -  (M^s \bar{R})_{J_i j}   \right| = \| (\breve M^s - M^s) \bar{R} \|_\infty.
\]
The first inequality now becomes
\begin{align*}
\left\| \tilde R^{(k)}  - \mathcal{R}^{(k)}  \right\|_\infty 
&\leq  1(k\geq 2) \sum_{t=1}^{k-1} \sum_{s=1}^t a_{s,t}  \| (\breve M^s - M^s) \bar{W} \|_\infty  +  \sum_{s=1}^k a_{s,k}  \| (\breve M^s  - M^s) \bar{R} \|_\infty \\
&\leq 1(k\geq 2) \sum_{t=1}^{k-1} \sum_{s=1}^t a_{s,t}  \| \breve M^s  - M^s \|_\infty \|  \bar W  \|_\infty  +  \sum_{s=1}^k a_{s,k}  \| \breve M^s  - M^s \|_\infty \| \bar{R} \|_\infty \\
&\leq \ell \sum_{t=1}^{k} \sum_{s=1}^t a_{s,t}  \| \breve M^s  - M^s \|_\infty,
\end{align*}
where in the last step we used the observation that $\| \bar W \|_\infty \leq  \ell$ and $\| \bar{R} \|_\infty \leq \ell$. For the remaining norm, note that  
\begin{align*}
x_s &:= \| \breve M^s  - M^s \|_\infty \leq \| \breve M^s  - \breve M^{s-1} M \|_\infty + \| \breve M^{s-1} M  - M^s \|_\infty \\
&\leq \| \breve M \|_\infty^{s-1} \| \breve M - M \|_\infty + \| \breve M^{s-1} - M^{s-1} \|_\infty \| M \|_\infty \leq \| \breve M - M \|_\infty  + x_{s-1} \\
&\leq s \| \breve M - M \|_\infty,
\end{align*}
since $x_0=0$ and $\|M\|_\infty=\|\breve{M}\|_\infty=1$.
Finally, note that for $\mathcal{I}$ the set of non-zero rows of $M$ (same as the non-zero rows of $\breve M$), we have
\begin{align*}
\| \breve M - M \|_\infty &= \max_{r \in \mathcal{I}} \sum_{s=1}^K  \left|  \frac{ \beta_{r, s}  \pi_s^{(n)} \kappa(s, r)}{ \beta_{r, 1} \pi_1^{(n)} \kappa(1,r) + \dots + \beta_{r, K} \pi_K^{(n)} \kappa(K, r)} -  \frac{ \beta_{r, s}  \pi_s \kappa(s, r)}{ \beta_{r, 1} \pi_1 \kappa(1,r) + \dots + \beta_{r, K} \pi_K \kappa(K, r)} \right| \\[5pt] &=
\max_{r \in \mathcal{I}} \dfrac{\sum_{s=1}^K \beta_{r,s} \kappa(s,r) \left| (\pi_s^{(n)} \pi_1 - \pi_s \pi_1^{(n)}) \beta_{r, 1}  \kappa(1,r) + \dots + (\pi_s^{(n)} \pi_K - \pi_s \pi_K^{(n)})  \beta_{r, K}  \kappa(K, r))  \right| }{(\beta_{r, 1} \pi_1^{(n)} \kappa(1,r) + \dots + \beta_{r, K} \pi_K^{(n)} \kappa(K, r) ) (  \beta_{r, 1} \pi_1 \kappa(1,r) + \dots + \beta_{r, K} \pi_K \kappa(K, r))} \\[5pt] &=
\max_{r \in \mathcal{I}} \dfrac{\sum_{s=1}^K\beta_{r,s}\pi_s\kappa(s,r)\sum_{l=1}^K\left| \frac{\pi_s^{(n)} \pi_l - \pi_s \pi_l^{(n)}}{\pi_l^{(n)} \pi_s} \right|\beta_{r,l}\pi_l^{(n)}\kappa(l,r)}{\sum_{m=1}^K\beta_{r,m}\pi_m^{(n)}\kappa(m,r)\sum_{n=1}^K\beta_{r,n}\pi_n\kappa(n,r)} \\[5pt] &\leq
\max_{1\leq r,s \leq K} \left| \frac{\pi_s^{(n)} \pi_r - \pi_s \pi_r^{(n)}}{\pi_r^{(n)} \pi_s} \right| = \mathcal{E}_n.
\end{align*}
Combining our estimates, we deduce
\begin{align*}
\left\| \tilde{R}^{(k)}  - \mathcal{R}^{(k)}   \right\|_\infty 
&\leq \ell \mathcal{E}_n\sum_{t=1}^{k} \sum_{s=1}^t a_{s,t} s = \ell c\,\mathcal{E}_n \sum_{t=1}^k (1-d)^{t-1} t  \leq \ell c\,\mathcal{E}_n\sum_{t=1}^\infty (1-d)^{t-1}t \leq \frac{\ell c}{d^2} \cdot \mathcal{E}_n ,
\end{align*}
where the last equality comes from \eqref{eq:BinMean}. Taking the supremum over $k$ completes the proof. 
\end{proof}

The proof of Theorem~\ref{MFA_generalthm} is based on the estimation in Theorem~\ref{T.interm_pnorm} of the expected norms
\begin{equation} \label{eq:InfinityNorms}
\mathbb{E}_n\left[ \left\| (C - \tilde M) \breve{X} \right\|_\infty \right] \qquad \text{and} \qquad \mathbb{E}_n\left[ \left\| C X- \tilde{M} \breve{X} \right\|_\infty \right] 
\end{equation}
for $X$ conditionally independent of $C$ given $\mathscr{F}_n = \sigma\left( \mathbf{Y}_i: i \in V_n \right)$ and having mean $\breve X = \mathbb{E}_n[X]$, when $\theta_n \geq \gamma \log n$, and of the expected norms
\begin{equation} \label{eq:OneNorms}
\max_{i \in V_n} \mathbb{E}_n\left[ \left\| (C^s X)_{i\bullet} - (\tilde{M}^s \breve{X})_{i\bullet} \right\|_1 \right] 
\end{equation}
for $s \geq 1$, when $\theta_n/\log n \to 0$ as $n \to \infty$. This is done in the next two sections.

\subsection{The semi-sparse to dense regimes}\label{SS.SemiDense}

When the edge density in the graph is sufficiently large, the concentration of the matrix $C$ around the matrix $\tilde M$ is very strong. It is worth pointing out that $\tilde M$ is not the mean of $C$, but rather, its more tractable asymptotic mean. For the norm involving $R^{(0)}$, the concentration occurs around its mean matrix $\bar{R}$, and the strength of this concentration is determined by the matrix $\tilde M$. The main result in this section provides estimates for the two expected norms in \eqref{eq:InfinityNorms}. Its proof relies on the following concentration inequality for Binomial or Poisson random sums, and its consequence on ratios of random sums.

\begin{prop} \label{P.NewConcentration}
Let $\{ Y_i^{(r)}: i \geq 1, 1\leq r \leq K\}$ be independent random variables on $[-H,H]$, with $\{ Y_i^{(r)}: i \geq 1\}$ i.i.d.~for each $1\leq r \leq K$, and $E[(Y_1^{(r)})^2] \leq v_r$. Let $\{ N_r: 1\leq r \leq K\}$ be independent random variables on $\mathbb{N}$, independent of $\{Y_i^{(r)}: i \geq 1, 1 \leq r \leq K\}$, and satisfying $E[e^{s N_r}] \leq e^{E[N_r] (e^s-1)}$ for $s \in \mathbb{R}$. Define $S_r = \sum_{i=1}^{N_r} Y_i^{(r)}.$
Then, for any $\epsilon>0$, 
\[ P\left(  \sum_{r=1}^K (S_r - E[S_r]) > \epsilon  \mu \right) \leq \text{exp}\left(  - \frac{( \epsilon \mu)^2}{2 \nu}   +   H   \frac{(\epsilon \mu)^3}{2 \nu^2}   \right)  , \]
where $\mu = \sum_{r=1}^K E[N_r] E[|Y_1^{(r)}|]$ and $\nu = \sum_{r=1}^K E[N_r] v_r$. 
\end{prop}

\begin{proof}
Start by conditioning on $\{N_r: 1\leq r \leq K\}$ and use Chernoff's bound to obtain that for any $\theta > 0$,
\begin{align*}
P\left( \sum_{r=1}^K (S_r - E[S_r]) > \epsilon \mu \right) &= E\left[  P\left( \left. \sum_{r=1}^K (S_r - E[S_r]) > \epsilon \mu \right| N_1, \dots, N_K \right) \right] \\
&\leq E\left[ e^{-\theta \epsilon \mu}  \prod_{r=1}^K E\left[ \left. e^{\theta (S_r- E[S_r])} \right| N_r \right] \right] = e^{-\theta \epsilon \mu}  \prod_{r=1}^K e^{-\theta E[S_r]}  E\left[ E\left[ e^{\theta Y_1^{(r)}}  \right]^{N_r}  \right] .
\end{align*}
Using the third order Taylor approximation for $e^{\theta x}$, we obtain 
\begin{align*}
E[ e^{\theta Y_1^{(r)}}] &\leq 1 + \theta E[Y_1^{(r)}] + \sum_{k=2}^\infty \frac{\theta^k E[|Y_1^{(r)}|^k]}{k!}  \leq 1 + \theta E[Y_1^{(r)}] + \left( \sum_{k=2}^\infty \frac{\theta^k H^{k-2}}{k!} \right)  E[(Y_1^{(r)})^2]  \\
&\leq 1 + \theta E[Y_1^{(r)}] +  \frac{v_r }{H^2} (e^{\theta H} -\theta H - 1) =: 1+ \lambda_r(\theta). 
\end{align*}
It follows from the assumption on the moment generating function of $N_r$ and the observation $E[S_r] = E[N_r] E[Y_1^{(r)}]$, that
\begin{align*}
P\left( \sum_{r=1}^K (S_r - E[S_r]) > \epsilon \mu \right) &\leq e^{-\theta \epsilon \mu} \prod_{r=1}^K e^{-\theta E[S_r]}   E\left[ (1+\lambda_r(\theta))^{N_r}  \right] = e^{-\theta \epsilon \mu}  \prod_{r=1}^K e^{-\theta E[S_r]} E\left[ e^{N_r\log(1+\lambda_r(\theta) )}  \right] \\
&\leq e^{-\theta \epsilon \mu}  \prod_{r=1}^K e^{ -\theta E[S_r] +  E[N_r] (e^{\log(1+\lambda_r(\theta))} -1) } = e^{-\theta \epsilon \mu +  \sum_{r=1}^K E[N_r] (\lambda_r(\theta) - \theta E[Y_1^{(r)}]  )} \\
&= e^{-\theta \epsilon \mu +   \frac{\nu}{H^2} (e^{\theta H} - \theta H - 1)} .
\end{align*}

Picking $\theta =  \frac{1}{H} \log\left(1 + H \epsilon \mu/\nu \right) \geq \epsilon \mu/\nu - (H \epsilon^2 \mu^2)/(2\nu^2)$, we get
\begin{align*}
P\left( \sum_{r=1}^K (S_r - E[S_r]) > \epsilon \mu \right) &\leq e^{-\theta \epsilon \mu + \frac{\nu}{H} ( \epsilon\mu/\nu-\theta )} \leq  e^{-\left(\frac{\epsilon \mu}{\nu} - \frac{H \epsilon^2 \mu^2}{2\nu^2} \right)(\epsilon \mu + \nu/H) +  \frac{\epsilon \mu}{H}  } = e^{-\frac{\epsilon^2 \mu^2}{2\nu}  + \frac{H \epsilon^3 \mu^3}{2\nu^2}  }.
\end{align*}
\end{proof}

Next, we use Proposition~\ref{P.NewConcentration} to bound the event that ratios of random sums are far from the ratios of their means. 

\begin{lemma} \label{L.Ratios}
Let $\{ B_i^{(r)}, X_i^{(r)}: i \geq 1, 1\leq r \leq K\}$ be independent random variables, with $B_i^{(r)} \in [0,H]$, $X_i^{(r)} \in [-1,1]$, and such that $\{ (B_i^{(r)}, X_i^{(r)}): i \geq 1\}$ are i.i.d.~for each $1\leq r \leq K$. Let $\{ N_r: 1\leq r \leq K\}$ be independent random variables on $\mathbb{N}$, independent of the $\{(B_i^{(r)}, X_i^{(r)}): i \geq 1, 1 \leq r \leq K\}$, and satisfying $E[e^{s N_r}] \leq e^{E[N_r] (e^s-1)}$ for $s \in \mathbb{R}$. Define $S_r = \sum_{i=1}^{N_r} B_i^{(r)}$ and $\tilde S_r = \sum_{i=1}^{N_r} X_i^{(r)} B_i^{(r)}$. Then, for any $\epsilon>0$, 
\[ P\left( \left|  \frac{\tilde S_1 + \dots + \tilde S_K}{S_1+\dots+ S_K} - \frac{ E[\tilde S_1] + \dots + E[\tilde S_K]}{E[S_1] + \dots + E[S_K]}   \right| > \epsilon \right) \leq 4 \text{exp}\left( - \frac{(\epsilon/2)^2 \mu^2 }{2 \nu} + \frac{H (\epsilon/2)^3 \mu^3}{2\nu^2}  \right), \]
where $\mu = \sum_{r=1}^K E[N_r] E[B_1^{(r)}]$ and $\nu = \sum_{r=1}^K E[N_r] E[(B_1^{(r)})^2]$. 
\end{lemma}

\begin{proof}
To start, note that since $|\tilde S_r| \leq S_r$ for each $r \in \{1, \dots, K\}$, then
\begin{align*}
&\left|  \frac{\tilde S_1 + \dots + \tilde S_K}{S_1+\dots+ S_K} - \frac{ E[\tilde S_1] + \dots + E[\tilde S_K]}{E[S_1] + \dots + E[S_K]}   \right|  \\
&=  \frac{1}{(S_1+\dots + S_K) E[S_1 + \dots + S_K]} \left| \sum_{r=1}^K  \sum_{t=1}^K E[S_r] \tilde S_t - \sum_{r=1}^K  \sum_{t=1}^K S_r E[ \tilde S_t] \right|  \\
&\leq  \frac{1}{(S_1+\dots + S_K) \mu} \left( \left| \sum_{r=1}^K  \sum_{t=1}^K E[S_r] \tilde S_t - \sum_{r=1}^K  \sum_{t=1}^K S_r \tilde S_t \right| + \left| \sum_{r=1}^K  \sum_{t=1}^K S_r \tilde S_t  - \sum_{r=1}^K  \sum_{t=1}^K S_r E[ \tilde S_t] \right|  \right) \\
&= \frac{1}{(S_1+\dots + S_K) \mu} \left( \left| \sum_{t=1}^K \tilde S_t \right| \left| \sum_{r=1}^K (E[S_r] - S_r)    \right| + \sum_{r=1}^K S_r \left|  \sum_{t=1}^K ( \tilde S_t - E[\tilde S_t]) \right|  \right) \\
&\leq   \frac{1}{ \mu} \left( \left| \sum_{r=1}^K (E[S_r] - S_r)    \right| +  \left|  \sum_{t=1}^K ( \tilde S_t - E[\tilde S_t]) \right|  \right) .
\end{align*} 
Next, define the events
\[ F = \left\{ \left| \sum_{r=1}^K (S_r - E[S_r]) \right| > \epsilon \mu/2 \right\} \qquad \text{and} \qquad G = \left\{ \left| \sum_{r=1}^K (\tilde S_r - E[\tilde S_r]) \right| > \epsilon \mu/2 \right\},   \]
and use the union bound to obtain that
\begin{align*}
P\left( \left|  \frac{\tilde S_1 + \dots + \tilde S_K}{S_1+\dots+ S_K} - \frac{ E[\tilde S_1] + \dots + E[\tilde S_K]}{E[S_1] + \dots + E[S_K]}   \right| > \epsilon \right) &\leq P\left(  \left| \sum_{r=1}^K (E[S_r] - S_r)    \right| +  \left|  \sum_{t=1}^K ( \tilde S_t - E[\tilde S_t]) \right|  > \epsilon \mu \right) \\[5pt]
&\leq P(F^c) + P(G^c ) .
\end{align*}
Finally, let $\tilde \mu = \sum_{r=1}^K E[N_r] E[|X_1^{(r)}|] E[B_1^{(r)}]$, and use Proposition~\ref{P.NewConcentration} with $Y_i^{(r)} = B_i^{(r)}$ and $Y_i^{(r)} = - B_i^{(r)}$ to obtain that
\[P(F^c) \leq 2 \text{exp}\left( - \frac{(\epsilon/2)^2 \mu^2}{2 \nu} + \frac{H (\epsilon/2)^3 \mu^3}{2 \nu^2}  \right),  \]
and then again with $Y_i^{(r)} = -X_i^{(r)} B_i^{(r)}$ and $Y_i^{(r)} = X_i^{(r)} B_i^{(r)}$ to get
\[ P(G^c) \leq 2  \text{exp}\left( - \frac{(\epsilon \mu/(2\tilde \mu))^2 \tilde\mu^2}{2 \nu} + \frac{H (\epsilon \mu/(2\tilde \mu))^3 \tilde \mu^3}{2 \nu^2}  \right) = 2 \text{exp}\left( - \frac{(\epsilon/2)^2 \mu^2 }{2 \nu} + \frac{H (\epsilon/2)^3 \mu^3}{2\nu^2}  \right).   \] 
This completes the proof.
\end{proof}

Now, we are able to prove the main result of this section. Recall that $\beta_{r,s} = E[ B_{ij} | J_i = r, J_j = s]$, $v_{r,s} = E[B_{ij}^2 | J_i = r, J_j = s]$ for $1 \leq r,s \leq K$,
\[ \mu_r^{(n)} = \sum_{s=1}^K \beta_{r,s} \pi_s^{(n)}\kappa(s,r) \qquad \text{and} \qquad \nu_r^{(n)} = \sum_{s=1}^K v_{r,s} \pi_s^{(n)}\kappa(s,r) , \]
for $1\leq r\leq K$ and large enough $n$, $\Delta_n = \max\limits_{r \in \mathcal{I}} \nu_r^{(n)}/(\mu_r^{(n)})^2$ and $\Lambda_n = \max\limits_{r \in \mathcal{I} } \mu_r^{(n)}/\nu_r^{(n)}$, where $\mathcal{I}$ is the set of nonzero rows of $M$.

\begin{theorem}\label{T.InfinityNorms} 
Let $X \in \mathbb{R}^{n\times \ell}$ be a random matrix conditionally independent of $C$ given $\mathscr{F}_n$ and having mean $\breve X = \mathbb{E}_n[X]$. Furthermore, suppose $X$ has conditionally independent rows $\{ \mathbf{X}_i: i \in V_n\}$ given $\mathscr{F}_n$ and satisfies $\| \mathbf{X}_i \|_\infty \leq 1$ a.s., and for each $1 \leq r \leq K$, its rows $\{ \mathbf{X}_i: J_i =r \}$ are  identically distributed.  Then, for $\theta_n \geq (2 H \Lambda_n/\gamma)^2 \Delta_n \log n$ and $0 < \gamma < 1/2$, there exists a constant $\Gamma' < \infty$ such that
\begin{align*}
\mathbb{E}_n\left[  \left\|  C X - \tilde M \breve X \right\|_\infty \right] &\leq \Gamma' \sqrt{\frac{\log n}{\theta_n}} ,\\
\mathbb{E}_n\left[ \left\| (C - \tilde M) \breve X \right\|_\infty \right]  &\leq \Gamma' \sqrt{\frac{\log n}{\theta_n}}.
\end{align*}
\end{theorem}

\begin{proof}
To start, note that since $C$ and $\tilde M$ have the same deterministically zero rows (if any), we can simply ignore them throughout the proof. Now let $\mathcal{I}_*$ denote the set of non-zero rows of $\tilde M$, fix $\epsilon > 0$, and define the events
\[ F_{it} = \left\{ \left| (C X)_{it} - (\tilde M \breve X)_{it} \right|  \leq 2 \epsilon \right\}, \quad i \in \mathcal{I}_*, \, 1\leq t \leq \ell. \]
Next, note that
\begin{align*}
\mathbb{E}_n\left[ \left\| CX - \tilde M \breve X \right\|_\infty \right] &= \mathbb{E}_n\left[ \max_{i \in \mathcal{I}_*} \sum_{t=1}^\ell  \left| (CX)_{it} - (\tilde M \breve X)_{it} \right| \right] \\
&\leq \mathbb{E}_n\left[ \max_{i \in \mathcal{I}_*} \sum_{t=1}^\ell  \left| (CX)_{it} - (\tilde M \breve X)_{it} \right| 1\left( F_{it} \right)  \right] \\
&\hspace{5mm} + \mathbb{E}_n\left[ \max_{i \in \mathcal{I}_*} \sum_{t=1}^\ell  \left( |(CX)_{it}| + |(\tilde M \breve X)_{it}| \right) 1\left(  F_{it}^c \right)  \right] \\
&\leq \mathbb{E}_n\left[ \max_{i \in \mathcal{I}_*} \sum_{t=1}^\ell 2\epsilon 1\left(  F_{it} \right)  \right] + \mathbb{E}_n\left[ \max_{i \in \mathcal{I}_*} 2\sum_{t=1}^\ell 1(F_{it}^c)  \right] \\
&\leq 2\ell \epsilon + 2 \sum_{i \in \mathcal{I}_*} \mathbb{E}_n\left[  \sum_{t=1}^\ell 1\left( F_{it}^c \right)  \right],
\end{align*}
where in the last step we used the observation that $$\sum_{t=1}^\ell |(CX)_{it}| 1(F_{it}^c) \leq \sum_{t=1}^\ell \sum_{j=1}^n C_{ij} |X_{jt}| 1(F_{it}^c) \leq  \sum_{t=1}^\ell \| C_{i\bullet} \|_1 1(F_{it}^c) \leq \sum_{t=1}^\ell 1(F_{it}^c) ,$$
and similarly, $\sum_{t=1}^\ell |(\tilde M \breve X)_{it}| 1(F_{it}^c) \leq \sum_{t=1}^\ell 1(F_{it}^c)$. It follows that
\[ \mathbb{E}_n\left[ \left\| CX - \tilde M \breve X \right\|_\infty \right] \leq 2\epsilon \ell + 2 \sum_{i \in \mathcal{I}_*} \sum_{t=1}^\ell \mathbb{P}_n \left( F_{it}^c \right) .  \]
To bound $\mathbb{P}_n\left( F_{it}^c \right)$, start by noting that
\begin{align*}
 \left| (C X)_{it} - (\tilde M \breve X)_{it} \right|  &= \left|\sum_{j=1}^n C_{ij} X_{jt} - \sum_{j=1}^n \tilde M_{ij} \breve X_{jt} \right| =   \left| \frac{\sum_{j=1}^n X_{jt} B_{ij}1(j \to i)}{\sum_{r=1}^n B_{ir} 1(r \to i)} - \frac{\sum_{j=1}^n \breve X_{jt} \beta_{J_i,J_j}p_{ji}^{(n)}}{\sum_{r=1}^n \beta_{J_i,J_r} p_{ri}^{(n)}} \right| .
\end{align*}
Next, note that we can rewrite
\[ \sum_{j=1}^n X_{jt} B_{ij}1(j \to i) = \sum_{r=1}^K \sum_{j=1}^n X_{jt} B_{i j}1(j \to i) 1(J_j = r) =: \sum_{r=1}^K  S_{i,r}^{(t)} \]
and
\[ \sum_{j=1}^n B_{ij}1(j \to i) = \sum_{r=1}^K \sum_{j=1}^n B_{i j}1(j \to i) 1(J_j = r) =: \sum_{r=1}^K S_{i,r}. \]
Moreover, conditionally on $J_i$, we have that $S_{i,r}^{(t)} \stackrel{\mathcal{D}}{=} \sum_{j=1}^{N_{i,r}} X_j^{(t,r)} B_j^{(i,r)}$ and $S_{i,r} \stackrel{\mathcal{D}}{=}  \sum_{j=1}^{N_{i,r}} B_j^{(i,r)}$, with $N_{i,r}$ a binomial with parameters $(n\pi_r^{(n)}, \kappa(r,J_i)\theta_n/n)$, independent of the i.i.d.~sequences $\{B_j^{(i,r)}: j \geq 1\}$ and $\{ X_j^{(t,r)}: j \geq 1\}$, which satisfy  $B_j^{(i,r)} \stackrel{\mathcal{D}}{=} (B_{ij} |J_i,  J_j = r )$, $X_j^{(t,r)} \stackrel{\mathcal{D}}{=} (X_{jt} | J_j = r)$, and $\bar x_j = \breve X_{jt}$ for $J_j = r$. It follows from Lemma~\ref{L.Ratios} that
\begin{align*}
\mathbb{P}_n\left( F_{it}^c \right) &= P\left( \left| \frac{S_{i,1}^{(t)}+ \dots +  S_{i,K}^{(t)}}{S_{i,1} + \dots + S_{i,K} } - \frac{E[S_{i,1}^{(t)}] + \dots + E[ S_{i,K}^{(t)}]}{E[S_{i,1}] + \dots + E[S_{i,K}]} \right| > 2\epsilon \right) \\
&\leq 4 \text{exp}\left( - \frac{\epsilon^2 (\theta_n \mu_{J_i}^{(n)})^2 }{2 \theta_n \nu_{J_i}^{(n)}} + \frac{H \epsilon^3 (\theta_n \mu_{J_i}^{(n)})^3}{2 (\theta_n\nu_{J_i}^{(n)})^2}  \right) = 4 \text{exp}\left( - \frac{\epsilon^2 \theta_n ( \mu_{J_i}^{(n)})^2 }{ 2 \nu_{J_i}^{(n)}} \left(1 -  \frac{H \epsilon \mu_{J_i}^{(n)}}{\nu_{J_i}^{(n)}} \right)  \right)
\end{align*}
where 
\begin{align*}
 \mu_{J_i}^{(n)} &= \sum_{r=1}^K \pi_r^{(n)} \kappa(r,J_i)  \beta_{J_i,r} = \theta_n^{-1}  \sum_{r=1}^K E[N_{i,r}] E[B_1^{(i,r)}] \\
\nu_{J_i}^{(n)} &=  \sum_{r=1}^K \pi_r^{(n)} \kappa(r,J_i)  v_{J_i,r} = \theta_n^{-1} \sum_{r=1}^K E[N_{i,r}] E[(B_1^{(r)})^2] .
\end{align*}

Therefore, 
\begin{align*}
\mathbb{E}_n\left[ \left\| CX - \tilde M \breve X \right\|_\infty \right] &\leq 2\ell \epsilon + 8 \sum_{i \in \mathcal{I}_*} \sum_{t=1}^\ell \text{exp}\left( - \frac{\epsilon^2 \theta_n ( \mu_{J_i}^{(n)})^2 }{ 2 \nu_{J_i}^{(n)}} \left(1 -  \frac{H \epsilon \mu_{J_i}^{(n)}}{\nu_{J_i}^{(n)}} \right)  \right) \\
&= 2\ell \epsilon + 8 \ell \sum_{r \in \mathcal{I}} n\pi_r^{(n)} \text{exp}\left( - \frac{\epsilon^2 \theta_n ( \mu_{r}^{(n)})^2 }{ 2 \nu_{r}^{(n)}} \left(1 -  \frac{H \epsilon \mu_{r}^{(n)}}{\nu_{r}^{(n)}} \right)  \right) \\
&\leq  2\ell \epsilon + 8 \ell \sum_{r \in \mathcal{I}} n\pi_r^{(n)} \text{exp}\left( - \frac{\epsilon^2 \theta_n ( \mu_{r}^{(n)})^2 }{ 2\nu_{r}^{(n)}} \left(1 -   H \epsilon \Lambda_n \right)  \right)  \\
&\leq 2\ell \epsilon + 8 \ell  n \, \text{exp}\left( - \frac{\epsilon^2 \theta_n  }{ 2 \Delta_n} \left(1 -  H \epsilon \Lambda_n \right)  \right),
\end{align*}
where $\Delta_n = \max_{r \in \mathcal{I}} \nu_r^{(n)}/(\mu_r^{(n)})^2$ and $\Lambda_n = \max_{r \in \mathcal{I}} \mu_r^{(n)}/\nu_r^{(n)}$. 
Next, choose $\epsilon = \sqrt{4 \Delta_n \log n/\theta_n}$ to obtain 
\begin{align*}
    n \, \text{exp}\left( - \frac{\epsilon^2 \theta_n  }{2  \Delta_n} \left(1 -  H \epsilon \Lambda_n \right)  \right) &= n\, \text{exp}\left( -  2\log n  \left( 1  -  \frac{2  H \Lambda_n \sqrt{ \Delta_n \log n} }{ \theta_n^{1/2} } \right)  \right) \\ &=
    \text{exp}\left( -  \log n  \left( 1  -  \frac{4  H \Lambda_n \sqrt{ \Delta_n \log n} }{ \theta_n^{1/2} } \right) \right) ,
\end{align*}
which in turn yields
\[ \mathbb{E}_n \left[   \left\| C X - \tilde M \breve X  \right\|_\infty  \right]  \leq 4 \ell  \sqrt{ \Delta_n \log n/\theta_n} + 8 \ell  \, \text{exp}\left( -  \log n    \left(1 -  \frac{4 H \Lambda_n \sqrt{ \Delta_n \log n} }{ \theta_n^{1/2} } \right)  \right) . \] 
Since the second norm corresponds to the case when $X$ is deterministic, the same argument also yields
\[ \mathbb{E}_n \left[   \left\| (C  - \tilde M) \breve X  \right\|_\infty  \right] \leq 4\ell  \sqrt{ \Delta_n \log n/\theta_n} + 8 \ell \, \text{exp}\left( -  \log n    \left(1 -  \frac{4 H \Lambda_n \sqrt{ \Delta_n \log n} }{ \theta_n^{1/2} } \right)  \right) . \]
To complete the proof, note that if $(2H \Lambda_n/\gamma)^2 \Delta_n \log n \leq \theta_n \leq (\log n)^3$, then 
\begin{align*}
    &\sqrt{ \Delta_n \log n/\theta_n}  +  \text{exp}\left( -  \log n    \left(1 -  \frac{4 H \Lambda_n \sqrt{ \Delta_n \log n} }{ \theta_n^{1/2} } \right)  \right) \\ 
    &\hspace{5mm} \leq \Delta_n^{1/2}\sqrt{\frac{\log n}{\theta_n}} + \exp\left( -\left( 1 -2\gamma \right)\log n  \right) =
    \Delta_n^{1/2}\sqrt{\frac{\log n}{\theta_n}}+n^{-1+2\gamma} = O\left( \sqrt{\frac{\log n}{\theta_n}} \right)
\end{align*}
as $n \to \infty$, while for $(\log n)^3 < \theta_n \leq n/ \kappa_*$, with $\kappa_* = \max_{1 \leq r,s\leq K} \kappa(r,s)$, 
\begin{align*}
    &\sqrt{ \Delta_n \log n/\theta_n}  +  \text{exp}\left( -  \log n    \left(1 -  \frac{4 H \Lambda_n \sqrt{ \Delta_n \log n} }{ \theta_n^{1/2} } \right)  \right) \\ 
    &\hspace{5mm}\leq \Delta_n^{1/2}\sqrt{\frac{\log n}{\theta_n}}+ n^{-1} e^{4H \Lambda_n \sqrt{\Delta_n}} = O\left( \sqrt{\frac{\log n}{\theta_n}} \right)
\end{align*}
as $n \to \infty$. The statement of the theorem now follows. 
%
%
\end{proof}

We can now give the proof to the first part of Theorem~\ref{MFA_generalthm}. 

\begin{proof}[Proof of Theorem~\ref{MFA_generalthm}]
Suppose $\theta_n \geq (6H \Lambda_n)^2 \Delta_n \log n$. We will prove that
\[\sup_{k \geq 0} \mathbb{E}_n\left[ \left\| R^{(k)} - \mathcal{R}^{(k)} \right\|_\infty \right] = O\left( \sqrt{\frac{\log n}{\theta_n} } + \mathcal{E}_n \right) \]
as $n \to \infty$, where $\mathcal{E}_n = \max_{1\leq r,s \leq K} \left| \frac{\pi_s^{(n)} \pi_r - \pi_s \pi_r^{(n)}}{\pi_r^{(n)} \pi_s} \right|$. To start, use the triangle inequality to obtain 
\begin{align*}
\sup_{k \geq 0} \mathbb{E}_n\left[ \left\| R^{(k)} - \mathcal{R}^{(k)} \right\|_\infty \right] &\leq \sup_{k \geq 0}  \mathbb{E}_n\left[ \left\| R^{(k)} - \tilde{R}^{(k)} \right\|_\infty \right] +  \mathbb{E}_n\left[ \sup_{k \geq 0} \left\| \tilde{R}^{(k)} - \mathcal{R}^{(k)} \right\|_\infty \right].
\end{align*}
Theorem~\ref{T.interm_meanfield} yields
$$\sup_{k \geq 0} \left\| \tilde{R}^{(k)} - \mathcal{R}^{(k)} \right\|_\infty \leq \frac{c\ell}{d^2}  \, \mathcal{E}_n.$$
For the remaining term, use Theorem~\ref{T.interm_pnorm} to obtain that
\begin{align*}
 \mathbb{E}_n\left[ \left\| R^{(k)} - \tilde{R}^{(k)} \right\|_\infty \right]  &\leq \max_{r\geq 1}  h( \tilde{M}^r \breve{W}  )  \sum_{t=0}^{k-1} (1-d)^{t-1} t + \mathbb{E}_n \left[  \left\| C W^{(0)} - \tilde M \breve{W} \right\|_\infty \right] \sum_{t=0}^{k-1} (1-d)^t  \\
&\hspace{5mm} + \max_{r\geq 1}  h( \tilde{M}^r \breve{R} ) \, k (1-d)^{k-1}   + \mathbb{E}_n \left[  \left\| C R^{(0)} - \tilde M \breve{R}  \right\|_\infty \right] (1-d)^k 
\end{align*}
where $h(X) = \mathbb{E}_n\left[ \left\| (C - \tilde M) X \right\|_\infty \right]$. Now use Theorem~\ref{T.InfinityNorms} to obtain that, since $\left\| \tilde{M}^r \breve{W} \right\|_\infty \leq \ell$ and $\left\| \tilde{M}^r \breve{R} \right\|_\infty \leq \ell$ for all $r \geq 0$, then
\begin{align*}
&\max\left\{ \max_{r \geq 1} h( \tilde{M}^r \breve{W} ), \, \max_{r \geq 1} h( \tilde{M}^r \breve{R} ), \,  \mathbb{E}_n \left[  \left\| C W^{(0)} - \tilde M \breve{W} \right\|_\infty \right] , \,  \mathbb{E}_n \left[  \left\| C R^{(0)} - \tilde M \breve{R} \right\|_\infty \right]  \right\} \leq \ell \Gamma' \sqrt{\frac{\log n}{\theta_n} }.
\end{align*}
Therefore,
\begin{align*}
 \mathbb{E}_n\left[ \left\| R^{(k)} - \tilde{R}^{(k)} \right\|_\infty \right]  &\leq \ell \Gamma'  \sqrt{\frac{\log n}{\theta_n} }  \left(\sum_{t=0}^{k-1} (1-d)^{t-1} t +  \sum_{t=0}^{k-1} (1-d)^t + k(1-d)^{k-1} + (1-d)^k \right),
 \end{align*}
and since the supremum over $k$ is finite, there exists some other constant $\Gamma < \infty$ such that
$$\sup_{k \geq 0} \left(  \mathbb{E}_n\left[ \left\| R^{(k)} - \tilde{R}^{(k)} \right\|_\infty \right] + \left\| \tilde{R}^{(k)} - \mathcal{R}^{(k)} \right\|_\infty \right)   \leq \Gamma \left(  \sqrt{\frac{\log n}{\theta_n} } + \mathcal{E}_n \right).$$
This completes the proof for the first part of the theorem.
\end{proof}

We now move to the density regime where $\log n/\theta_n$ may not converge to zero. In that case, we can still prove concentration of the opinion matrix $R^{(k)}$ to the mean-field limit $\mathcal{R}^{(k)} $ as $n\to\infty$, but in a weaker mode of convergence.

\subsection{The subcritical to sparse regimes} \label{SS.SemiSparse}

In this section we complete the proof of Theorem~\ref{MFA_generalthm} by analyzing the expected norms \eqref{eq:OneNorms}. When $\limsup_{n \to \infty} (\log n)/\theta_n > 0$, the number of neighbors is not enough to provide a full concentration of the matrix $C$ around its asymptotic mean $\tilde M$, since up to an $o(n)$ number of vertices will not be sufficiently close to their limits. Instead of using the supremum norm as in the denser case, we use an averaged $l_1$-norm, which is what we typically use in sparse graphs where the local limiting structure of the dSBM is what enables the analysis. The key observation in this section is that for any matrix $X \in \mathbb{R}^{n\times\ell}$,  $(C^s X)_{ij}$ is a function of the inbound neighborhood of vertex $i \in V_n$ of length $s$, which is approximately equal to $(\tilde{M}^s X)_{ij}$. In order for us to make this connection, we need to describe the local behavior of neighborhoods in the marked graph $G_n =(V_n, E_n; \mathscr{A}_n)$, where $\mathscr{A}_n = \{ \mathbf{Y}_i: i \in V_n\}$ are the extended vertex attributes.

\subsubsection{Coupling with a multi-type Galton-Watson tree} \label{CouplingPBT}

Pick a vertex $I_n$, chosen uniformly at random from the vertex set $V_n$. The idea behind local weak convergence is to couple the exploration of the inbound neighborhood of $I_n$ with the construction of a rooted graph, whose root $\emptyset$ corresponds to $I_n$. We point out that in our setting, the local weak limit of the dSBM is a non locally-finite graph, so the coupling we need is one between the exploration of the inbound neighborhood of $I_n$ with a branching tree whose degree distribution grows with $n$. We start by introducing  some terminology for trees.

We enumerate nodes in a rooted tree using the Ulam-Harris notation, which contains the ancestry path of any node all the way to the root. For that, let any node in generation $k \geq 1$ of the tree have a label of the form $\bi=(i_1, \dots, i_k)\in\mathbb{N}_+^k$,  and define the concatenation operation $(\bi, j)=(i_1, \dots, i_k, j)$ for a node in generation $k+1$. The root is always denoted by $\emptyset$, but is omitted from the label of any other node; to simplify the notation, we omit the parentheses for nodes in the first generation. We denote by $\mathcal{U}= \{ \emptyset \} \cup \bigcup_{k=1}^{\infty}\mathbb{N}_+^k$ the set of all possible node labels. The generation of a node $\bi\in\mathcal{U}$ is denoted by $|\bi|$, for $\bi \neq \emptyset$. 

The local weak limit of a sparse dSBM with $K$ communities is a marked $K$-type Galton-Watson tree. Recall that our network is not a sparse graph, so instead we are using an \textit{intermediate} coupling (see \cite{olvera2022strong} for a detailed description of the construction). In order to obtain a description of the coupled tree, we start by reminding the reader that vertices in $G_n$ have extended attributes:
\begin{equation} \label{eq:FullMark}
\mathbf{Y}_i = (J_i, \mathbf{Q}_i, \mathbf{B}_i) \in \{ 1, \dots, K\} \times [-1,1]^\ell \times [0,H]^n,
\end{equation}
and $\mathscr{F}_n = \sigma\left( \mathbf{Y}_i: i \in V_n \right)$. In the $K$-type Galton-Watson tree, each node $\bi$ will have an {\em attribute}:
\[ \boldsymbol{\hat A}_\bi = \left(  \hat J_\bi, \boldsymbol{\hat Q}_\bi \right) \in \{1, \dots, K\} \times [-1,1]^\ell, \]
and a {\em full mark} of the form:
$$\boldsymbol{\hat X}_\bi = \left( \hat J_\bi, \boldsymbol{\hat Q}_\bi, \{ \hat N_\bi^{(r)} \}_{r=1}^K, \{ \hat B_{(\bi, s)} \}_{s=1}^{\hat N_\bi} \right) \in \{1, \dots, K\} \times [-1,1]^\ell \times \mathbb{N}^K \times [0, H]^\infty,$$
where $\hat J_\bi$ is the type of node $\bi$, $\hat N_\bi^{(r)}$ is the number of type $r$ offspring that node $\bi$ has, and $\hat N_\bi = \hat N_\bi^{(1)} + \dots + \hat N_\bi^{(K)}$.  

We will construct the tree conditionally on the community labels $\mathscr{J}_n = \{J_i: i \in V_n\}$. The unmarked tree will be denoted $\mathcal{T} = \bigcup_{k=0}^\infty \mathcal{A}_k \subset \mathcal{U}$, where $\mathcal{A}_k$ is the set of nodes in its $k$th generation, and each node $\bi \in \mathcal{T}$ will have a type $\hat J_\bi$. To start, set $\mathcal{A}_0 = \{ \emptyset \}$ and choose the type of the root $\emptyset$ according to:
\[ \mathbb{P}_n\left(\hat{J}_{\emptyset}=r\right)=\dfrac{1}{n}\sum\limits_{i=1}^n1(J_i=r)=\pi_r^{(n)}. \]
From here on, any node $\bi \in \mathcal{A}_k$, will have an offspring vector $(\hat N_\bi^{(1)}, \dots, \hat N_\bi^{(K)})$, conditionally independent of everything else given its type, and having distribution:
\[ \mathbb{P}_n\left(\hat N_\bi^{(1)} = n_1, \dots, \hat N_\bi^{(K)}= n_K\,\big\vert\,\hat{J}_\bi=r \right) = \prod\limits_{s=1}^K e^{-q_{sr}^{(n)} }\dfrac{(q_{sr}^{(n)} )^{n_s}}{n_s!}, \]
where $q_{sr}^{(n)} = \kappa(s,r) \pi_s^{(n)} \theta_n$. In other words, given $\hat J_\bi=r$, the numbers of offspring of each type $s$ of a node of type $r$ are independent Poisson random variables with means $q_{sr}^{(n)}$, $1\leq s \leq K$. To prevent the labels from conveying any information, we usually permute the $\hat N_\bi$ offspring uniformly at random, and then set 
\[\mathcal{A}_{k+1} = \left\{ (\bi, j): \bi \in \mathcal{A}_k, 1 \leq j \leq \hat N_\bi \right\}. \]
Note that the conditional distribution of $\hat J_{(\bi,j)}$ for $|\bi| = k$, $k \geq 0$, depends on $\hat J_\bi$ and is given by
\[\mathbb{P}_n\left( \left. \hat J_{(\bi,j)} = s \right| \hat J_\bi = r \right) = \mathbb{E}_n\left[ \left. \frac{\hat N_\bi^{(s)} 1(\hat N_\bi > 0)}{\hat N_\bi} \right| \hat J_\bi =r \right] = \frac{q_{sr}^{(n)}}{\sum_{t=1}^K q_{tr}^{(n)}} = \frac{\kappa(s,r) \pi_s^{(n)}}{\sum_{t=1}^K \kappa(t,r) \pi_t^{(n)}}. \]
Moreover, conditionally on $\hat J_\bi$, the $\{ \hat J_{(\bi,j)}: j \geq 1\}$ are independent of each other. 


Now that the unmarked tree $\mathcal{T} = \bigcup_{k=0}^\infty \mathcal{A}_k$ has been constructed, we assign to each node $\bi \in \mathcal{T}$ a mark $\left(\boldsymbol{\hat Q}_\bi, \{ \hat B_{(\bi, s)} \}_{s=1}^{\hat N_\bi} \right)$. The internal belief vector $\boldsymbol{\hat Q}_\bi$ depends only on $\hat J_\bi$ and is distributed according to $F_{\hat J_\bi}$. Finally, the unnormalized weights satisfy for $x \in \mathbb{R}$,
\begin{equation} \label{eq:TreeWeights}
 \mathbb{P}_n\left( \left. \hat B_{(\bi,j)} \leq x \right|  \hat J_\bi, \hat J_{(\bi,j)} \right) = G_{\hat J_\bi, \hat J_{(\bi,j)} } ( x), 
 \end{equation}
independently of everything else. 

\begin{defin} \label{D.Isomorphism}
    Two simple vertex-marked directed graphs, $G=(V, E; \mathscr{A})$ and $G'=(V', E'; \mathscr{A}')$, whose vertices have marks $\mathbf{a}_i$, $i \in V$, and $\mathbf{a}_i'$, $i \in V'$, respectively, are said to be \textit{isomorphic} if there exists a bijection $\vartheta: V\rightarrow V'$ such that $(i, j)\in E$ if, and only if, $(\vartheta(i), \vartheta(j))\in E'$, for nodes $i, j\in V$, and for all $i \in V$ the vertex mark $\mathbf{a}_i = \mathbf{a}_{\vartheta(i)}'$. We write $G\simeq G'$ to express that the graphs $G$ and $G'$ are isomorphic.
\end{defin}

For each $i \in V_n$, the intermediate coupling will construct a marked tree $\mathcal{T}_{\emptyset(i)}(\hat J) \stackrel{\mathcal{D}}{=} (\mathcal{T}(\hat J) | \hat J_\emptyset = J_i)$ using as marks in Definition~\ref{D.Isomorphism} the types, which correspond to the community labels in the dSBM (we do not need the internal opinions nor the unnormalized weights for this purpose, since their distribution is determined by the types). Note that the subscript $\emptyset(i)$ emphasizes that the root of the tree corresponds to vertex $i$ in the graph. For any fixed $s \geq 0$, let $\mathcal{G}_i^{(s)}(J)$ denote the induced subgraph of $G_n$ determined by the vertices in $V_n$ that have directed paths of length at most $s$ connecting them to vertex $i$, including their community labels, and let $\mathcal{T}_{\emptyset(i)}^{(s)}(\hat J)$ denote the restriction of $\mathcal{T}_{\emptyset(i)}(\hat J)$ to its first $s$ generations. Then, using the same type of intermediate coupling as in \cite{lee2020pagerank}, one obtains the following theorem.

\begin{theorem} \label{T.Coupling}
For the dSBM under Assumption~\ref{regularity_condition_communities}, with $\theta_n \leq (\log n)^b$ for some $b > 0$,  and any fixed $k \geq 0$, there exists a sequence of marked $K$-type Galton-Watson trees $\{ \mathcal{T}_{\emptyset(i)}(\hat J): i \in V_n\}$ such that $\mathcal{T}_{\emptyset(i)}(\hat J) \stackrel{\mathcal{D}}{=} (\mathcal{T}(\hat J) | \hat J_\emptyset = J_i)$ and
\[ \frac{1}{n} \sum_{i=1}^n \mathbb{P}_n\left( \mathcal{G}_i^{(k)}(J) \not\simeq \mathcal{T}_{\emptyset(i)}^{(k)}(\hat J) \right) \xrightarrow{P} 0, \qquad n \to \infty. \]
\end{theorem}

Note that when the coupling holds for a vertex $i$ up to generation $s$, i.e.,  $\mathcal{G}_i^{(s)}(J) \simeq \mathcal{T}_{\emptyset(i)}^{(s)}(\hat J)$, we can copy the internal belief vectors and the unnormalized weights from the graph, so $\boldsymbol{\hat Q}_\bi = \mathbf{Q}_{\vartheta(\bi)}$ and $\hat B_{(\bi,j)} = B_{\vartheta(\bi), \vartheta((\bi,j))}$, where $\vartheta$ is the bijection defining the isomorphism. When the coupling fails, we simply sample them independently according to $F_{\hat J_\bi}$ and \eqref{eq:TreeWeights}, respectively. Once we have the unnormalized weights, we define
$$\hat C_{(\bi,j)} = \frac{\hat B_{(\bi,j)}}{\sum_{s=1}^{\hat N_\bi} \hat B_{(\bi,s)}} \cdot 1\left( \sum_{s=1}^{\hat N_\bi} \hat B_{(\bi,s)} > 0\right).$$

On the tree, the weights $\{\hat{C}_\bi: \bi\in \mathcal{T}\}$ are used to construct the weights
\[ \Pi_{(\bi, j)}= \Pi_\bi\hat{C}_{(\bi,j)}, \qquad \Pi_\emptyset=1. \]

\subsubsection{Proof of Theorem~\ref{MFA_generalthm} for semi-sparse graphs}

Our goal is to bound the norm
\[
\max_{i \in V_n} \mathbb{E}_n\left[ \left\| (C^s X)_{i\bullet} - (\tilde{M}^s \breve{X})_{i\bullet} \right\|_1 \right] = \max_{i \in V_n} \sum_{j=1}^t \mathbb{E}_n\left[ \left| (C^s X)_{ij} - (\tilde M^s \breve X)_{ij} \right| \right]
\]
for $X \in [-1,1]^{n \times \ell}$ a random matrix, conditionally independent of $C$ given $\mathscr{F}_n$ and having mean $\breve X = \mathbb{E}_n[X]$. 

To simplify the notation, note that if we set $\mathbf{X} = X_{\bullet j}$ and $\mathbf{\bar x} = \breve X_{\bullet j} =  \mathbb{E}_n[ \mathbf{X}]$, then
\[\mathbb{E}_n\left[ \left| (C^s X)_{ij} - (\tilde M^s \breve X)_{ij} \right| \right] = \mathbb{E}_n\left[ \left| (C^s \mathbf{X})_{i} - (\tilde M^s \mathbf{\bar x})_{i} \right| \right], \]
so it suffices to analyze expectations of this form. We start by defining for each $i \in V_n$ and any $s \geq 1$, the event that the coupling with the $K$-type Galton-Watson tree holds:
\begin{align*}
D_{s,i} = \left\{ \mathcal{G}_i^{(s)}(J) \simeq \mathcal{T}_{\emptyset(i)}^{(s)}(\hat J) \right\}.
\end{align*}

\begin{lemma} \label{L.TreeWeights}
Let $\mathbf{X} = (X_1, \dots, X_n)^\top \in [-1,1]^n$ be a random vector such that $X_i$ has distribution $H_{\mathbf{A}_i}$, independently of everything else; let $\mathbf{\bar x} =  (\bar x_1, \dots, \bar x_n)^\top = \mathbb{E}_n[\mathbf{X}]$ denote its mean.  Furthermore, suppose $\bar x_i = x_r$ for all $i$ such that $J_i =r$, $1\leq r \leq K$. Then, for any $i \in V_n$ and $s \geq 1$,
\[ \mathbb{E}_n\left[  \left| (C^s \mathbf{X})_i - (\tilde M^s \mathbf{\bar x})_i \right| \right] \leq \mathbb{E}_n\left[ \left. \left| \sum_{|\bi| = s} \Pi_\bi \hat X_\bi  - (\breve M \mathbf{x})_{J_i} \right| \right| \hat J_\emptyset = J_i \right] + 2 \mathbb{P}_n\left(  D_{s,i}^c  \right), \]
where $\breve M$ is defined in \eqref{eq:BreveM}, $\mathbf{x} = (x_1, \dots, x_K)^\top$, the $\Pi_\bj$ are defined on the marked tree $\mathcal{T}(\hat J)$, and $\hat X_\bi$ is distributed according to $H_{\boldsymbol{\hat A}_\bi}$, independently of everything else. 
\end{lemma}

\begin{proof}
We start by considering one vertex $i \in V_n$ and construct its coupled tree $\mathcal{T}_{\emptyset(i)}(\hat J)  \stackrel{\mathcal{D}}{=} ( \mathcal{T}(\hat J) | \hat J_\emptyset = J_i)$. Note that on the event $D_{s,i}$ there is a unique path connecting vertex $i$ to each of its neighbors at distance $s$, and these paths are the same as those in the tree, so
\begin{equation}\label{CsX_ingeneral}
   (C^s\mathbf{X})_i = \sum_{j_1=1}^n \cdots \sum_{j_s = 1}^n C_{ij_1}C_{j_1j_2}\cdots C_{j_{s-1}j_s} X_{j_s} = \sum_{|\bi| = s} \Pi_\bi \hat X_\bi,
\end{equation}
with $\hat X_\bi = X_{\vartheta(\bi)}$, with $\vartheta$ the bijection determining the event $D_{s,i}$. Since the distribution of $X_{\vartheta(\bi)}$ depends only on $\mathbf{A}_{\vartheta(\bi)}$ and $\mathbf{A}_{\vartheta(\bi)} = \boldsymbol{\hat A}_\bi$, then $(\hat X_\bi | \boldsymbol{\hat A}_\bi)$ is distributed according to $H_{\boldsymbol{\hat A}_\bi}$. 

The induction proof given within the proof of Theorem~\ref{T.interm_meanfield} shows that
\[(\tilde M^s \mathbf{\bar x})_i = (\breve M \mathbf{x})_{J_i}. \]
Consequently,
\begin{align*}
\mathbb{E}_n\left[  \left| (C^s \mathbf{X})_i - (\tilde M^s \mathbf{\bar x})_i \right| \right] &=  \mathbb{E}_n\left[  \left| (C^s \mathbf{X})_i - (\tilde M^s \mathbf{\bar x})_i  \right| \right] \\
&\leq  \mathbb{E}_n\left[  1(D_{s,i}) \left| (C^s \mathbf{X})_i - (\tilde M^s \mathbf{\bar x})_i  \right|  \right] +  \mathbb{E}_n\left[  1(D_{s,i}^c) \left( |(C^s \mathbf{X})_i | + | (\tilde M^s \mathbf{\bar x})_i | \right)  \right] \\
&\leq  \mathbb{E}_n\left[ \left. \left| \sum_{|\bi| = s} \Pi_\bi \hat X_\bi  - (\breve M \mathbf{x})_{J_i} \right| \right| \hat J_\emptyset = J_i \right] + 2 \mathbb{P}_n\left(  D_{s,i}^c  \right),
\end{align*}
where the weights $\Pi_\bi$ and the terminal values $\hat X_\bi$ are defined on the marked tree $\mathcal{T}_{\emptyset(i)}(\hat J)$, and we have used the observation that $\| C^s \mathbf{X} \|_\infty \leq 1$ and $\| \tilde M^s \mathbf{\bar x} \|_\infty \leq 1$. 
\end{proof}

Next, define the vector $\mathbf{a}_s := (a_s(1),\dots,a_s(K))^{\top}$ given by:
\begin{equation} \label{eq:Arecursion}
a_s(r) := \mathbb{E}_n\left[ \left. \left| \sum_{|\bi|=s}\hat{\Pi}_\bi\hat{X}_\bi-(\breve{M}^s\mathbf{x})_r \right| \right| \hat{J}_\emptyset= r \right]. 
\end{equation}

Before proving the main theorem of this section, we give a bound for $\| \mathbf{a}_1 \|_\infty$.

\begin{lemma} \label{L.OneGeneration}
For $\mathbf{a}_1$ defined according to \eqref{eq:Arecursion}, and provided $\theta_n\to \infty$ as $n \to \infty$, there exists a constant $\Gamma' < \infty$ that does not depend on the choice of $\{\hat X_i\}$ such that 
\[ \| \mathbf{a}_1 \|_\infty  \leq \Gamma' \theta_n^{-1/2} . \] 
\end{lemma}

\begin{proof}
Note that if the $r$th row of $\breve M$ is zero, then $a_1(r) = 0$, and since the non-zero rows of $\breve M$ coincide with those of $M$, it suffices to consider $r \in \mathcal{I}$. For $r \in \mathcal{I}$, 
\[ a_1(r) = \mathbb{E}_n\left[ \left. \left| \sum_{i=1}^{\hat N_\emptyset} \hat C_i \hat{X}_i-(\breve{M}\mathbf{x})_r \right| \right| \hat{J}_\emptyset= r \right], \]
and let 
\[ S_{j} := \sum_{i=1}^{\hat N_\emptyset} \hat B_i \hat X_i 1(\hat J_i = j) \qquad \text{and} \qquad \tilde S_{j} := \sum_{i=1}^{\hat N_\emptyset} \hat B_i \hat X_i 1(\hat J_i = j), \]
so that 
\[ \sum_{i=1}^{\hat N_\emptyset} \hat C_i \hat{X}_i = \sum_{i=1}^{\hat N_\emptyset} \frac{\hat B_i}{\sum_{s=1}^{\hat N_\emptyset} \hat B_s}  \hat{X}_i = \frac{\tilde S_1 + \dots + \tilde S_K}{S_1 + \dots + S_K}. \]
Define the event
\[F = \left\{ \left| \frac{ \tilde S_1 + \dots + \tilde S_K}{S_1 + \dots + S_K} - (\breve{M} \mathbf{x})_{\hat J_\emptyset} \right| \leq 2/(3H\Lambda_n)  \right\}, \]
and condition on $\hat J_\emptyset$ to obtain that
\begin{align*}
&\mathbb{E}_n\left[ \left. \left| \sum_{i=1}^{\hat N_\emptyset} \hat C_i \hat{X}_i-(\breve{M}\mathbf{x})_r \right| \right| \hat J_\emptyset = r \right] \\
&\leq \mathbb{E}_n\left[ \left. 1(F) \left| \sum_{i=1}^{\hat N_\emptyset} \hat C_i \hat{X}_i-(\breve{M}\mathbf{x})_r \right| \right| \hat J_\emptyset = r \right] + \mathbb{E}_n\left[ \left. 1(F^c) \left( \sum_{i=1}^{\hat N_\emptyset} \hat C_i |\hat{X}_i| + |(\breve{M}\mathbf{x})_r| \right) \right| \hat J_\emptyset = r \right]   \\
&= \int_0^{2/(3H\Lambda_n)} \mathbb{P}_n\left( \left. \left| \sum_{i=1}^{\hat N_\emptyset} \hat C_i \hat{X}_i-(\breve{M}\mathbf{x})_r \right| > t \right| \hat J_\emptyset = r\right) dt + 2\,\mathbb{P}_n\left( F^c \left| \hat J_\emptyset = r \right. \right)  ,
\end{align*}
where we used the observation that $\sum_{i=1}^{\hat N_\emptyset} \hat C_i | \hat X_i| \leq 1$ and $|(\breve M \mathbf{x})_r| \leq \| \breve M \mathbf{x} \|_\infty \leq \| \breve M \|_\infty \| \mathbf{x} \|_\infty \leq 1$, and therefore, the integrand is bounded by one. 

Next, note that conditionally on $\hat J_\emptyset =r$, we have that $\tilde S_j \stackrel{\mathcal{D}}{=} \sum_{i=1}^{N_j} B_i^{(r,j)} X_i^{(j)}$, where $N_j$ is a Poisson random variable with mean $q_{jr}^{(n)} = \kappa(j,r) \pi_j^{(n)} \theta_n$, and is independent of the i.i.d.~sequences $\{ B_i^{(r,j)}: i \geq 1\}$ and $\{ X_i^{(j)}: i \geq 1\}$, with $B_i^{(r,j)} \stackrel{\mathcal{D}}{=} (\hat B_i |\hat J_\emptyset = r, \hat J_i = j)$ and $X_i^{(j)} \stackrel{\mathcal{D}}{=} (\hat X_i | \hat J_i = j)$.  Similarly, $S_j \stackrel{\mathcal{D}}{=} \sum_{i=1}^{N_r} B_i^{(r,j)}$, and all the random variables $\{ S_{j}, \tilde S_j: 1 \leq j \leq K\}$ are conditionally independent of each other given $\hat J_\emptyset$. It follows from Lemma~\ref{L.Ratios} that for $r \in \mathcal{I}$, 
\begin{align*}
\mathbb{P}_n\left( \left. \left| \sum_{i=1}^{\hat N_\emptyset} \hat C_i \hat{X}_i-(\breve{M}\mathbf{x})_r \right| > t \right| \hat J_\emptyset =r \right) &\leq 4 \, \text{exp}\left( -\frac{(t/2)^2 (\theta_n \mu_r^{(n)})^2}{2 \theta_n \nu_r^{(n)}} + \frac{H (t/2)^3 (\theta_n \mu_r^{(n)})^3}{2 ( \theta_n \nu_r^{(n)})^2} \right) \\
&\leq 4 \,  \text{exp}\left( -\frac{(t/2)^2 \theta_n }{2 \Delta_n} \left( 1 - H \Lambda_n t/2 \right)  \right) ,
\end{align*}
as in the proof of Theorem \ref{T.InfinityNorms},
and
\[ \mathbb{P}_n\left( F^c \left| \hat J_\emptyset = r \right. \right)  \leq 4 \,  \text{exp}\left( -\frac{(1/(3H\Lambda_n))^2 \theta_n }{2 \Delta_n} \left( 1 - H \Lambda_n /(3H\Lambda_n) \right)  \right) = 4 \, \text{exp} \left( -\frac{\theta_n}{27 \Delta_n ( H \Lambda_n)^2} \right). \]
Since the bounds do not depend on $\hat J_\emptyset$, we obtain that
\begin{align*}
 \| \mathbf{a}_1 \|_\infty &\leq 4 \int_0^{2/(3H\Lambda_n)} \text{exp}\left( -\frac{(t/2)^2 \theta_n }{2 \Delta_n} \left( 1 - H \Lambda_n t/2 \right)  \right) dt + 8 \, \text{exp} \left( -\frac{\theta_n}{27 \Delta_n ( H \Lambda_n)^2} \right) \\
 &\leq 4  \int_0^{2/(3H\Lambda_n)} \text{exp}\left( -\frac{(t/2)^2 \theta_n }{3 \Delta_n}   \right) dt + 8 \, \text{exp} \left( -\frac{\theta_n}{27 \Delta_n ( H \Lambda_n)^2} \right) \\
 &= \frac{4}{\sqrt{\theta_n/(6\Delta_n)}} \int_0^{2 \sqrt{\theta_n/(6\Delta_n)}/ (3H\Lambda_n)} e^{-z^2/2} dz + 8 \, \text{exp} \left( -\frac{\theta_n}{27 \Delta_n ( H \Lambda_n)^2} \right) \\
 &\leq \frac{8 \sqrt{3 \Delta_n \pi}  }{\sqrt{\theta_n}} \int_0^\infty \frac{e^{-z^2/2}}{\sqrt{2\pi}} dz  + 8 \, \text{exp} \left( -\frac{\theta_n}{27 \Delta_n ( H \Lambda_n)^2} \right) \\
 &= \frac{4 \sqrt{3\Delta_n \pi}}{\sqrt{\theta_n}} + 8 \, \text{exp} \left( -\frac{\theta_n}{27 \Delta_n ( H \Lambda_n)^2} \right) = O\left( \theta_n^{-1/2} \right)
 \end{align*}
as $n \to \infty$. Noting that the last estimates do not depend on the choice of the $\{ \hat X_i\}$ completes the proof. 
\end{proof}

As a corollary to the previous lemma, we obtain the following bound for $\| \mathbf{a}_s \|_\infty $.

\begin{cor}\label{C.Recursive}
For $s \geq 1$ and $\mathbf{a}_s$ defined according to \eqref{eq:Arecursion}, and provided $\theta_n\to \infty$ as $n \to \infty$, there exists a constant $\Gamma' < \infty$ that does not depend on the choice of $\{\hat X_i\}$ such that 
 \[ \| \mathbf{a}_s \|_\infty \leq   \Gamma' s \theta_n^{-1/2} . \]  
\end{cor}

\begin{proof}
To derive a recursion for $\mathbf{a}_s$, $s \geq 1$, note that for each $j\in\{1,\dots,K\}$,
\begin{align*}
a_s(j) &= \mathbb{E}_n\left[ \left. \left| \sum_{r=1}^{\hat N_\emptyset} \hat C_r \sum_{|(r,\bi)|=s} \frac{\hat{\Pi}_{(r,\bi)}}{\hat C_{r}} \hat{X}_{(r,\bi)} - \sum_{r=1}^{\hat N_\emptyset} \hat C_r (\breve M^{s-1} \mathbf{x})_{\hat J_r} +  \sum_{r=1}^{\hat N_\emptyset} \hat C_r (\breve M^{s-1} \mathbf{x})_{\hat J_r} -(\breve{M}^s\mathbf{x})_j \right| \right| \hat{J}_\emptyset=j \right] \\
&\leq \mathbb{E}_n\left[ \left.\sum_{r=1}^{\hat N_\emptyset} \hat C_r \left| \sum_{|(r,\bi)|=s} \frac{\hat{\Pi}_{(r,\bi)}}{\hat C_{r}} \hat{X}_{(r,\bi)} - (\breve M^{s-1} \mathbf{x})_{\hat J_r}  \right| \right| \hat{J}_\emptyset=j \right]  \\
&\hspace{5mm} + \mathbb{E}_n\left[ \left. \left|   \sum_{r=1}^{\hat N_\emptyset} \hat C_r (\breve M^{s-1} \mathbf{x})_{\hat J_r} -(\breve{M}^s\mathbf{x})_j \right| \right| \hat{J}_\emptyset=j \right] \\
&= \mathbb{E}_n\left[ \left.\sum_{r=1}^{\hat N_\emptyset} \hat C_r \mathbb{E}_n\left[ \left. \left| \sum_{|(r,\bi)|=s} \frac{\hat{\Pi}_{(r,\bi)}}{\hat C_{r}} \hat{X}_{(r,\bi)} - (\breve M^{s-1} \mathbf{x})_{\hat J_r}  \right| \right| \hat N_\emptyset, \{ \hat C_r, \hat J_r \}_{r=1}^{\hat N_\emptyset} \right] \right| \hat{J}_\emptyset=j \right]  \\
&\hspace{5mm} + \mathbb{E}_n\left[ \left. \left|   \sum_{l=1}^K \sum_{r=1}^{\hat N_\emptyset} \hat C_r 1(\hat J_r = l)  (\breve M^{s-1} \mathbf{x})_{l} - \sum_{l=1}^K \breve{M}_{jl} (\breve{M}^{s-1} \mathbf{x})_l  \right| \right| \hat{J}_\emptyset=j \right] \\
&= \mathbb{E}_n\left[ \left.\sum_{r=1}^{\hat N_\emptyset} \hat C_r a_{s-1}(\hat{J}_r) \right| \hat{J}_\emptyset=j \right]   + \mathbb{E}_n\left[ \left. \left|   \sum_{l=1}^K  (\breve M^{s-1} \mathbf{x})_{l} \left( \sum_{r=1}^{\hat N_\emptyset} \hat C_r 1(\hat J_r = l) -   \breve{M}_{jl} \right)  \right| \right| \hat{J}_\emptyset=j \right]  \\
&\leq \max_{1 \leq j \leq K} a_{s-1}(j) \mathbb{E}_n\left[ \left. \sum_{r=1}^{\hat N_\emptyset} \hat C_r \right| \hat J_\emptyset = j \right] + \mathbb{E}_n\left[ \left. \left|   \sum_{l=1}^K  (\breve M^{s-1} \mathbf{x})_{l} \left( \sum_{r=1}^{\hat N_\emptyset} \hat C_r 1(\hat J_r = l) -   \breve{M}_{jl} \right)  \right| \right| \hat{J}_\emptyset=j \right] . 
\end{align*}
To simplify the bound, recall that $\sum_{r=1}^{\hat N_\emptyset} \hat C_r = 1$. To bound the second expectation, let $\hat Y_r = y_l =  (\breve M^{s-1} \mathbf{x})_l$ for all $r$ such that $\hat J_r = l$, define $\mathbf{y} = (y_1, \dots, y_K)^\top$, and note that 
\[  \mathbb{E}_n\left[ \left. \left|   \sum_{l=1}^K  (\breve M^{s-1} \mathbf{x})_{l} \left( \sum_{r=1}^{\hat N_\emptyset} \hat C_r 1(\hat J_r = l) -   \breve{M}_{jl} \right)  \right| \right| \hat{J}_\emptyset=j \right]  =  \mathbb{E}_n\left[ \left. \left|   \sum_{r=1}^{\hat N_\emptyset} \hat C_r \hat Y_r  -   (\breve{M} \mathbf{y})_j   \right| \right| \hat{J}_\emptyset=j \right] , \]
so by Lemma~\ref{L.OneGeneration} we get 
\[ \max_{1 \leq j \leq K} \mathbb{E}_n\left[ \left. \left|   \sum_{r=1}^{\hat N_\emptyset} \hat C_r \hat Y_r  -   (\breve{M} \mathbf{y})_j   \right| \right| \hat{J}_\emptyset=j \right] \leq  \Gamma' s \theta_n^{-1/2}   \]
for some constant $\Gamma'$ that does not depend on the $\{\hat Y_i\}$. Therefore, 
\[ \| \mathbf{a}_s \|_\infty \leq \| \mathbf{a}_{s-1} \|_\infty +  \Gamma' \theta_n^{-1/2} . \]
Iterating the recursion gives
\[ \| \mathbf{a}_s \|_\infty \leq  \Gamma' s \theta_n^{-1/2} . \]
\end{proof}

We now state the counterpart of Theorem~\ref{T.InfinityNorms} for this regime. 

\begin{theorem} \label{T.OneNorms}
Let $\mathbf{X} = (X_1, \dots, X_n)^\top \in [-1,1]^n$ be a random vector such that $X_i$ has distribution $H_{\mathbf{A}_i}$, independently of everything else; let $\mathbf{\bar x} =  (\bar x_1, \dots, \bar x_n)^\top = \mathbb{E}_n[\mathbf{X}]$ denote its mean.  Furthermore, suppose $\bar x_i = x_r$ for all $i$ such that $J_i =r$, $1\leq r \leq K$. Then, there exists a constant $\Gamma'$ that does not depend on the choice of $\mathbf{X}$ such that for any $i \in V_n$ and $s \geq 1$, 
\[  \mathbb{E}_n\left[ \left| (C^s \mathbf{X})_i - (\tilde M^s \mathbf{\bar x})_i \right| \right] \leq \frac{\Gamma' s}{\theta_n^{1/2}} +  2 \mathbb{P}_n\left( D_{s,i}^c \right). \]
\end{theorem}

\begin{proof}
By Lemma~\ref{L.TreeWeights} and Corollary~\ref{C.Recursive}, we have that
\begin{align*}
  \mathbb{E}_n\left[ \left| (C^s \mathbf{X})_i - (\tilde M^s \mathbf{\bar x})_i \right| \right] &\leq \mathbb{E}_n\left[ \left. \left| \sum_{|\bi| = s} \Pi_\bi \hat X_\bi  - (\breve M \mathbf{x})_{J_i} \right| \right| \hat J_\emptyset = J_i \right] + 2  \mathbb{P}_n\left(  D_{s,i}^c  \right) \\
 &\leq  a_s(J_i)  + 2 \mathbb{P}_n\left(  D_{s,i}^c  \right) \leq  \| \mathbf{a}_s \|_\infty + 2 \mathbb{P}_n\left(  D_{s,i}^c  \right) \\
 &\leq  \Gamma' s \theta_n^{-1/2} + 2  \mathbb{P}_n\left(  D_{s,i}^c  \right). 
 \end{align*}
\end{proof}

We can now give the proof to the second part of Theorem~\ref{MFA_generalthm}. 

\begin{proof}[Proof of Theorem~\ref{MFA_generalthm}]
Suppose $\theta_n \to \infty$ as $n \to \infty$, and recall that we need to prove that
\[ \sup_{k \geq 0} \, \max_{i \in V_n} \mathbb{E}_n \left[ \left\| \mathbf{R}_i^{(k)} - \boldsymbol{\mathcal{R}}_i^{(k)} \right\|_1 \right] \xrightarrow{P} 0 \]
as $n \to \infty$. Note that by Remark~\ref{R.Density}(a), it suffices to assume that $\theta_n \leq (\log n)^b$ for some $b > 1$. 
To start, note that by Theorem~\ref{T.interm_pnorm}, we have that for any fixed $k \geq 0$,
\begin{align*}
 \max_{i \in V_n} \, \mathbb{E}_n\left[ \left\|  \mathbf{R}_i^{(k)} - \boldsymbol{\mathcal{R}}_i^{(k)} \right\|_1 \right] &\leq  \max_{i \in V_n}  \, \mathbb{E}_n\left[ \left\|  \mathbf{R}_i^{(k)} - \mathbf{\tilde R}_i^{(k)} \right\|_1 \right] +  \max_{i \in V_n} \, \mathbb{E}_n\left[ \left\|  \mathbf{\tilde R}_i^{(k)} - \boldsymbol{\mathcal{R}}_i^{(k)} \right\|_1 \right]  \\
  &\leq 1(k\geq 2) \sum_{t=1}^{k-1} \sum_{s=1}^t a_{s,t} \max_{i \in V_n} \, \mathbb{E}_n\left[ \left\| (C^s W^{(0)})_{i\bullet} - (\tilde M^s \breve W)_{i\bullet} \right\|_1 \right] \\
  &\hspace{5mm} + \sum_{s=1}^k a_{s,k} \max_{i \in V_n} \, \mathbb{E}_n\left[ \left\| (C^s R^{(0)})_{i\bullet} - (\tilde M^s \breve R)_{i\bullet} \right\|_1 \right] + \mathbb{E}_n\left[  \left\| \tilde R^{(k)} - \mathcal{R}^{(k)} \right\|_\infty \right].
\end{align*}
Moreover, by Theorem~\ref{T.interm_meanfield},
\[ \mathbb{E}_n\left[  \left\| \tilde R^{(k)} - \mathcal{R}^{(k)} \right\|_\infty \right] \leq \mathbb{E}_n\left[ \sup_{k \geq 0} \left\| \tilde R^{(k)} - \mathcal{R}^{(k)} \right\|_\infty \right] \leq \frac{\ell c}{d^2} \cdot \mathcal{E}_n. \]

Now use Theorem~\ref{T.OneNorms} to obtain that
\[ \mathbb{E}_n\left[ \left\| (C^s W^{(0)})_{i\bullet} - (\tilde M^s \breve W)_{i\bullet} \right\|_1 \right] = \sum_{j=1}^\ell \mathbb{E}_n\left[ \left\| (C^s W^{(0)}_{\bullet j})_{i} - (\tilde M^s \breve W_{\bullet j})_{i} \right\|_1 \right] \leq \ell \left( \Gamma' s \theta_n^{-1/2} + 2 \mathbb{P}_n(D_{s,i}^c) \right)  \]
and
\[ \mathbb{E}_n\left[ \left\| (C^s R^{(0)})_{i\bullet} - (\tilde M^s \breve R)_{i\bullet} \right\|_1 \right] \leq \ell \left( \Gamma' s \theta_n^{-1/2} + 2 \mathbb{P}_n(D_{s,i}^c) \right).  \]
It follows that
\begin{align*}
    \max_{i \in V_n} \, \mathbb{E}_n\left[ \left\|  \mathbf{R}_i^{(k)} - \boldsymbol{\mathcal{R}}_i^{(k)} \right\|_1 \right] &\leq 1(k\geq 2) \sum_{t=1}^{k-1} \sum_{s=1}^t a_{s,t} \max_{i \in V_n} \,  \ell \left( \Gamma' s \theta_n^{-1/2} + 2 \mathbb{P}_n(D_{s,i}^c) \right) \\
  &\hspace{5mm} + \sum_{s=1}^k a_{s,k} \max_{i \in V_n} \,  \ell \left( \Gamma' s \theta_n^{-1/2} + 2 \mathbb{P}_n(D_{s,i}^c) \right) + \frac{\ell c}{d^2} \cdot \mathcal{E}_n \\
  &= \ell \Gamma' \theta_n^{-1/2}  \sum_{t=1}^{k} \sum_{s=1}^t a_{s,t} s + 2  \sum_{t=1}^{k} \sum_{s=1}^t a_{s,t} \max_{i \in V_n} \mathbb{P}_n(D_{s,i}^c) + \frac{\ell c}{d^2} \cdot \mathcal{E}_n \\
  &= \ell \Gamma' \theta_n^{-1/2} c \sum_{t=1}^k (1-d)^{t-1}t  + 2  \sum_{t=1}^{k} \sum_{s=1}^t a_{s,t} \max_{i \in V_n} \mathbb{P}_n(D_{s,i}^c) + \frac{\ell c}{d^2} \cdot \mathcal{E}_n,
\end{align*}
where in the last equality we used \eqref{eq:BinMean}. Since $\sum_{t=1}^k (1-d)^{t-1} t \leq d^{-2}$ for all $k \geq 0$, and 
\begin{align*}
\sum_{t=1}^{k} \sum_{s=1}^t a_{s,t} \max_{i \in V_n} \mathbb{P}_n(D_{s,i}^c) &=  \sum_{s=1}^k \max_{i \in V_n} \mathbb{P}_n(D_{s,i}^c) c^s \sum_{t=s}^{k}  \binom{t}{s} (1-c-d)^{t-s}  \\
&=  \sum_{s=1}^k \max_{i \in V_n} \mathbb{P}_n(D_{s,i}^c) c^s (c+d)^{-s-1} \sum_{m=0}^{k-s}  \binom{m+s}{m} (1-c-d)^{m} (c+d)^{s+1} \\
&\leq \sum_{s=1}^\infty \max_{i \in V_n} \mathbb{P}_n(D_{s,i}^c) \frac{c^s}{(c+d)^{s+1}} =  \frac{c}{(c+d)d} \mathbb{E}_n\left[ \max_{i \in V_n} \mathbb{P}_n(D_{T,i}^c) \right],
\end{align*}
where $T$ is a geometric random variable on $\{1,2,\dots\}$ with success probability $p = d/(c+d)$, we have shown that
\[ \sup_{k\geq 0} \max_{i \in V_n} \, \mathbb{E}_n\left[ \left\|  \mathbf{R}_i^{(k)} - \boldsymbol{\mathcal{R}}_i^{(k)} \right\|_1 \right]  \leq \frac{\ell c}{d^2} \left(\Gamma' \theta_n^{-1/2} + \mathcal{E}_n \right) + \frac{2c}{(c+d)d} \mathbb{E}_n\left[ \max_{i \in V_n} \mathbb{P}_n(D_{T,i}^c) \right] .  \]

It only remains to show that $\max_{i \in V_n} \mathbb{P}_n(D_{T,i}^c) \xrightarrow{P} 0$ as $n \to \infty$ for any fixed $T$, since then the bounded convergence theorem would complete the proof. To see this is the case, note that all vertices $\{i \in V_n: J_i = r\}$, $1\leq r \leq K$, are exchangeable, so for any $i \in V_n$,
\begin{align*}
     \mathbb{P}_n(D_{T,i}^c) &= \sum_{r=1}^K \mathbb{P}_n(D_{T,i}^c) 1(J_i = r) =  \sum_{r=1}^K \frac{1}{\pi_r^{(n)} n} \sum_{m=1}^n \mathbb{P}_n(D_{T,m}^c) 1(J_m = r) = \frac{1}{n} \sum_{m=1}^n \mathbb{P}_n(D_{T,m}^c) \sum_{r=1}^K \frac{1(J_m = r)}{\pi_r^{(n)}},
\end{align*}
and therefore, by Theorem~\ref{T.Coupling} and Assumption~\ref{regularity_condition_communities},
\[ \max_{i \in V_n} \mathbb{P}_n(D_{T,i}^c) \leq  \frac{1}{n} \sum_{m=1}^n \mathbb{P}_n(D_{T,m}^c) \cdot \sum_{r=1}^K\frac{1}{\pi_r^{(n)}} \xrightarrow{P} 0 \]
 as $n\to \infty$. This completes the proof.
\end{proof}

The last section of the paper contains the proofs of Theorems~\ref{thm_LWCprob} and \ref{thm_convergence_in_prob}, which are a consequence of Theorem~\ref{MFA_generalthm}.

\subsection{Proofs of Theorems~\ref{thm_LWCprob} and \ref{thm_convergence_in_prob} } \label{SS.LWCprob}

We start by proving Theorem~\ref{thm_LWCprob}, which focuses on the trajectories of the process $\{ R^{(k)}: k \geq 0\}$.

\begin{proof}[Proof of Theorem~\ref{thm_LWCprob}]
Fix $k \geq 0$ and recall that $V_{k,i} \in [-1,1]^{\ell \times (k+1)}$ is the matrix having $j$th column $(V_{k,i})_{\bullet j} = \mathbf{R}_i^{(j-1)}$, and $\mathcal{V}_k \in [-1,1]^{\ell \times (k+1)}$ is the matrix having $j$th column $(\mathcal{V}_k)_{\bullet j} = \boldsymbol{\mathcal{R}}_\emptyset^{(j-1)}$. Assume $f: [-1,1]^{\ell \times (k+1)} \to \mathbb{R}$ is bounded and continuous with respect to $\| \cdot \|_1$ (the induced operator norm) . Since $f$ is continuous on a compact set, it is also uniformly continuous, so for every $\epsilon > 0$ there exists a $\delta > 0$ such that
\[ \left| f( X )-f( Y )\right| < \varepsilon \qquad
\text{ whenever } \qquad  \| X - Y \|_1 < \delta. \]
Let $\mathcal{V}_{k,i} \in [-1,1]^{\ell \times (k+1)}$ denote the matrix having $j$th column $(\mathcal{V}_{k,i})_{\bullet j} = \boldsymbol{\mathcal{R}}_i^{(j-1)}$, and let 
\[f^* = \max_{X \in [-1,1]^{\ell \times (k+1)}} |f(X)|.\]
It follows that for any $1 \leq r \leq K$,
\begin{align*}
   & \dfrac{1}{n}\sum_{i=1}^n \left| f\left( V_{k,i} \right) - f\left( \mathcal{V}_{k,i} \right)\right| 1(J_i = r)  \\ 
    &\hspace{5mm}\leq \varepsilon\cdot\dfrac{1}{n}\sum_{i=1}^n1\left( \| V_{k,i} - \mathcal{V}_{k,i} \|_1 <\delta \right)  + 2 f^*\cdot \dfrac{1}{n}\sum_{i=1}^n1\left( \| V_{k,i} - \mathcal{V}_{k,i} \|_1 \geq\delta\right) \\ 
    &\hspace{5mm}\leq \varepsilon+ \dfrac{2 f^*}{\delta} \cdot\dfrac{1}{n}\sum_{i=1}^n \| V_{k,i} - \mathcal{V}_{k,i} \|_1\\
    &\hspace{5mm}= \varepsilon+ \dfrac{2 f^*}{\delta} \cdot \frac{1}{n}\sum_{i=1}^n \max_{0\leq m\leq k} \left\| \mathbf{R}_i^{(m)} - \boldsymbol{\mathcal{R}}_i^{(m)} \right\|_1 \\ 
    &\hspace{5mm}\leq \varepsilon+ \dfrac{2 f^*}{\delta} \cdot \dfrac{1}{n} \sum_{m=0}^k\sum_{i=1}^n \left\|\textbf{R}_i^{(m)}-\boldsymbol{\mathcal{R}}_i^{(m)}\right\|_1 
\end{align*}
\noindent
Take expectations on both sides to obtain 
\begin{align*}
    \mathbb{E}_n\left[ \dfrac{1}{n}\sum_{i=1}^n \left| f\left( V_{k,i} \right) - f\left( \mathcal{V}_{k,i} \right)\right| 1(J_i = r) \right] &\leq \varepsilon+ \dfrac{2 f^*}{\delta} \cdot \dfrac{1}{n} \sum_{m=0}^k\sum_{i=1}^n \mathbb{E}_n\left[ \left\|\textbf{R}_i^{(m)}-\boldsymbol{\mathcal{R}}_i^{(m)}\right\|_1 \right] \\ &\leq 
    \varepsilon+ \dfrac{2 f^*}{\delta} \cdot \sum_{m=0}^k \max_{i\in V_n}\mathbb{E}_n\left[ \left\|\textbf{R}_i^{(m)}-\boldsymbol{\mathcal{R}}_i^{(m)}\right\|_1 \right] \\ &\leq 
    \varepsilon+\frac{2(k+1)f^*}{\delta}\cdot\sup_{m\geq0}\max_{i\in V_n}\mathbb{E}_n\left[ \left\|\textbf{R}_i^{(m)}-\boldsymbol{\mathcal{R}}_i^{(m)}\right\|_1\right].
\end{align*}
From Theorem \ref{MFA_generalthm}, we have that $\sup_{m\geq0}\max_{i\in V_n}\mathbb{E}_n\left[ \left\|\textbf{R}_i^{(m)}-\boldsymbol{\mathcal{R}}_i^{(m)}\right\|_1\right] \xrightarrow{P} 0$ as $n \to \infty$. The bounded convergence theorem then gives
$$\lim_{n \to \infty}  E\left[ \dfrac{1}{n}\sum_{i=1}^n \left| f\left( V_{k,i} \right) - f\left( \mathcal{V}_{k,i} \right)\right| 1(J_i = r) \right] \leq
\varepsilon+\frac{2(k+1)f^*}{\delta}\cdot \lim_{n \to \infty} E\left[\sup_{m\geq0}\max_{i\in V_n}\mathbb{E}_n\left[ \left\|\textbf{R}_i^{(m)}-\boldsymbol{\mathcal{R}}_i^{(m)}\right\|_1\right]\right]=\varepsilon,$$
which in turn implies that
$$\dfrac{1}{n}\sum_{i=1}^n \left| f\left( V_{k,i} \right) - f\left( \mathcal{V}_{k,i} \right)\right| 1(J_i = r) \xrightarrow{P} 0, \qquad n \to \infty.$$
Moreover, Kolmogorov's strong law of large numbers gives
\[ \left| \frac{1}{n} \sum_{i=1}^n \left( f(\mathcal{V}_{k,i}) 1(J_i=r)  - E\left[f\left( \mathcal{V}_{k,i} \right) 1(J_i= r) \right]  \right) \right| \to 0 \qquad a.s. \]
as $n \to \infty$. 
Since$$E[ f(\mathcal{V}_{k,i}) 1(J_i = r)] = E[f(\mathcal{V}_{k,i}) | J_i = r] E[\pi_r^{(n)}] = E[ f(\mathcal{V}_k) | \mathcal{J}_\emptyset = r] E[ \pi_r^{(n)}] = E[ f(\mathcal{V}_k) 1(\mathcal{J}_\emptyset = r)] E[\pi_r^{(n)}]/\pi_r ,$$ 
we have shown that
\begin{align*}
&\left| \frac{1}{n} \sum_{i=1}^n f(V_{k,i}) 1(J_i = r) - E\left[f\left( \mathcal{V}_k \right) 1(\mathcal{J}_\emptyset = r) \right] \right|  \\
&\hspace{5mm}\leq  \dfrac{1}{n}\sum_{i=1}^n \left| f\left( V_{k,i} \right) - f\left( \mathcal{V}_{k,i} \right)\right|1(J_i = r) + \left| \frac{1}{n} \sum_{i=1}^n f(\mathcal{V}_{k,i}) 1(J_i=r)  - E\left[f\left( \mathcal{V}_k \right) 1(\mathcal{J}_\emptyset = r) \right] \right|  \\
&\hspace{5mm}\leq \dfrac{1}{n}\sum_{i=1}^n \left| f\left( V_{k,i} \right) - f\left( \mathcal{V}_{k,i} \right)\right|1(J_i = r) + \left| \frac{1}{n} \sum_{i=1}^n \left( f(\mathcal{V}_{k,i}) 1(J_i=r)  - E\left[f\left( \mathcal{V}_{k,i} \right) 1(J_i= r) \right]  \right) \right| \\
&\hspace{10mm} + \left| E[ f(\mathcal{V}_k) 1(\mathcal{J}_\emptyset = r)] \right|  \left|  \frac{E[\pi_r^{(n)}]}{\pi_r} -   1  \right| \xrightarrow{P} \varepsilon
\end{align*}
as $n \to \infty$. Since $\varepsilon > 0$ was arbitrary, the first statement of the theorem follows.

For the second statement, let $\{f_1, \dots, f_m\}$ be bounded and continuous (with respect to $\| \cdot \|_1$) functions on $[-1,1]^{\ell \times (k+1)}$. Recall that we may assume that the functions are uniformly continuous. Now note that $$E\left[ \left. f_j\left( \mathcal{V}_k \right) \right| \mathcal{J}_\emptyset = J_{i_j} \right] = \mathbb{E}_n\left[ f_j\left( \mathcal{V}_{k,i_j} \right) \right],$$ and since the $\{\mathcal{V}_{k,i}: i \in V_n\}$ are conditionally independent given the community labels $\mathscr{J}_n$, then
\begin{align*}
\left| \mathbb{E}_n \left[ \prod_{j=1}^m f_j\left( V_{k,i_j} \right) \right] - \prod_{j=1}^m E\left[ \left. f_j\left( \mathcal{V}_k \right) \right| \mathcal{J}_\emptyset = J_{i_j} \right] \right| &= \left| \mathbb{E}_n \left[ \prod_{j=1}^m f_j\left( V_{k,i_j} \right) \right] - \mathbb{E}_n\left[  \prod_{j=1}^m f_j\left(\mathcal{V}_{k,i_j} \right) \right] \right| \\
&\leq  \mathbb{E}_n \left[ \left| \prod_{j=1}^m f_j\left( V_{k,i_j} \right) -  \prod_{j=1}^m f_j\left(\mathcal{V}_{k,i_j} \right)  \right| \right]  \\
&\leq  \prod_{j=1}^m f_j^*  \sum_{j=1}^m \mathbb{E}_n \left[   \left|  f_j\left( V_{k,i_j} \right) -  f_j\left(\mathcal{V}_{k,i_j} \right)  \right| \right] ,
\end{align*}
where $f_j^* = \max_{X \in [-1,1]^{\ell \times (k+1)}} |f_j(X)|$, and we have used the inequality $$\left| \prod_{i=1}^m x_i - \prod_{i=1}^m y_i \right| \leq \prod_{i=1}^m (|x_i| \vee |y_i|) \sum_{i=1}^m |x_i - y_i|.$$
To complete the proof, note that the random variables $\{  \left|  f_j\left( V_{k,t} \right) -  f_j\left(\mathcal{V}_{k,t} \right)  \right|: J_t = r \}$ are exchangeable, so if $J_{i_j} = r_j$, then
\begin{align*}
\mathbb{E}_n \left[   \left|  f_j\left( V_{k,i_j} \right) -  f_j\left(\mathcal{V}_{k,i_j} \right)  \right| \right]  &= \frac{1}{n\pi_{r_j}^{(n)}} \sum_{t: J_t = r_j} \mathbb{E}_n \left[   \left|  f_j\left( V_{k,t} \right) -  f_j\left(\mathcal{V}_{k,t} \right)  \right| \right] \\
&= \frac{1}{\pi_{r_j}^{(n)}}  \mathbb{E}_n \left[ \frac{1}{n} \sum_{t=1}^n  \left|  f_j\left( V_{k,t} \right) -  f_j\left(\mathcal{V}_{k,t} \right)  \right| 1(J_t = r_j) \right]  .
\end{align*}
Now apply the same arguments used in the first part of the theorem to obtain that for any $\varepsilon >0$ there exists a $\delta > 0$ such that
\[ \mathbb{E}_n \left[   \left|  f_j\left( V_{k,i_j} \right) -  f_j\left(\mathcal{V}_{k,i_j} \right)  \right| \right]   \leq  \frac{1}{\pi_{r_j}^{(n)}} \left(\varepsilon + \frac{2(k+1)f_j^*}{\delta}\sup_{m\geq0}\max_{i\in V_n}\mathbb{E}_n \left[\left\| \mathbf{R}_i^{(m)}-\boldsymbol{\mathcal{R}}_i^{(m)} \right\|_1 \right]\right).  \]
We conclude that
\begin{align*}
&\left| \mathbb{E}_n \left[ \prod_{j=1}^m f_j\left( V_{k,i_j} \right) \right] - \prod_{j=1}^m E\left[ \left. f_j\left( \mathcal{V}_k \right) \right| \mathcal{J}_\emptyset = J_{i_j} \right] \right|  \\
&\hspace{5mm}\leq \prod_{j=1}^m f_j^* \sum_{j=1}^m \frac{1}{ \pi_{r_j}^{(n)} } \left( \varepsilon + \frac{2(k+1)f_j^*}{\delta}\sup_{m\geq0}\max_{i\in V_n} \mathbb{E}_n \left[\left\| \mathbf{R}_i^{(m)}-\boldsymbol{\mathcal{R}}_i^{(m)} \right\|_1\right] \right) \xrightarrow{P}  \prod_{j=1}^m f_j^* \sum_{j=1}^m \frac{\varepsilon}{ \pi_{r_j} }  
\end{align*}
as $n \to \infty$. Taking $\varepsilon \downarrow 0$ completes the proof. 
\end{proof}

We now prove Theorem~\ref{thm_convergence_in_prob}, which refers to the stationary behavior of the process $\{ R^{(k)}: k \geq 0\}$.

\begin{proof}[Proof of Theorem \ref{thm_convergence_in_prob}] 
We will start by constructing a coupling of two matrices, $R \in [-1,1]^{n \times \ell}$ and $\mathcal{R} \in [-1,1]^{n \times \ell}$, the first one with rows $R_{j \bullet} = \mathbf{R}_j$ and the second one with rows $\mathcal{R}_{j \bullet} = \boldsymbol{\mathcal{R}}_j$, 
as well as of the random variable $\mathcal{J}_\emptyset$, such that $R$ has the stationary distribution of $\{ R^{(k)}: k \geq 0\}$,
$$\boldsymbol{\mathcal{R}}_i = \sum_{t=0}^{\infty} (1-c-d)^t \textbf{W}^{(t)}_i +  \sum_{t=1}^{\infty} \sum_{s=1}^t a_{s,t} (M^s \bar W)_{J_i}  , \qquad i \in V_n,$$
and 
\[ \max_{i\in V_n}\mathbb{E}_n \left[\left\| \mathbf{R}_i-\boldsymbol{\mathcal{R}}_i \right\|_1 \right]\to0, \qquad n\to\infty. \]
Note that exactly the same arguments used in the proof of Theorem~\ref{thm_LWCprob} (set $k = 0$) will then yield
$$\frac{1}{n} \sum_{i=1}^n f(\textbf{R}_i) 1(J_i = r) \xrightarrow{P} E[ f(\boldsymbol{\mathcal{R}}_\emptyset) 1(\mathcal{J}_\emptyset = r) ] \qquad \text{and} \qquad \mathbb{E}_n\left[ \prod_{j=1}^m f_j(\mathbf{R}_{i_j}) \right] \xrightarrow{P} \prod_{j=1}^m E[ f_j(\boldsymbol{\mathcal{R}}_\emptyset) | \mathcal{J}_\emptyset = J_{i_j}],$$
as $n \to \infty$. Hence, we only need to exhibit the coupling.

Define $\{ B^{(k)}: k \geq 0\} \subset [-1,1]^{n\times\ell}$ to be the process defined according to 
$$B^{(k)} = \sum_{t=0}^{k-1} A^t W^{(t)},$$
and let $B = \sum_{t=0}^\infty A^t W^{(t)}$. To see that $B$ is well-defined, recall that $\| A \|_\infty = 1-d<1$, and therefore, 
\[ 
\left\| B-B^{(k)}\right\|_\infty = \left\| \sum_{t=k}^\infty A^tW^{(t)} \right\|_\infty \leq \sum_{t=k}^\infty\|A\|_\infty^t \| W^{(t)} \|_\infty \leq \ell\sum_{t=k}^\infty (1-d)^t = \ell (1-d)^k/d \to 0,
 \]
as $k \to \infty$, which yields $B = \lim_{k \to \infty} B^{(k)}$. 

Going back to the original opinion recursion (\ref{opinion_recursion_matrix_form}), note that it satisfies
$$R^{(k)} = \sum_{r=0}^{k-1} A^t W^{(k-t)} + A^k R^{(0)} \stackrel{\mathcal{D}}{=} B^{(k)} +  A^k R^{(0)} ,$$
from where we obtain
$$R^{(k)} \Rightarrow B, \qquad k \to \infty.$$
It follows that $B \stackrel{\mathcal{D}}{=} R$. The coupling we need is given by $(B, \mathcal{R})$. It remains to show that 
$$\max_{i\in V_n}\mathbb{E}_n \left[\left\| \mathbf{B}_i-\boldsymbol{\mathcal{R}}_i \right\|_1 \right]\to0, \qquad n\to\infty,$$
where $\mathbf{B}_i$ is the $i$th row of $B$.

To this end, define the matrix $\mathcal{B}^{(k)} \in [-1,1]^{n \times \ell}$ having rows $(\mathcal{B}^{(k)})_{i \bullet } = \boldsymbol{\mathcal{B}}_i^{(k)}$ given by
$$\boldsymbol{\mathcal{B}}^{(k)}_i =  \sum_{t=0}^{k-1} (1-c-d)^t \textbf{W}^{(t)}_i + 1(k \geq 2) \sum_{t=1}^{k-1} \sum_{s=1}^t a_{s,t} (M^s \bar W)_{J_i} .$$
Next, note that for any $k \geq 1$ and any $i\in V_n$, 
\begin{align*}
    \left\| \mathbf{B}_i-\boldsymbol{\mathcal{R}}_i \right\|_1 &\leq \left\| \mathbf{B}_i-\mathbf{B}_i^{(k)} \right\|_1 +
    \left\| \mathbf{B}_i^{(k)}-\boldsymbol{\mathcal{B}}_i^{(k)} \right\|_1 +
    \left\| \boldsymbol{\mathcal{B}}_i^{(k)}-\boldsymbol{\mathcal{R}}_i \right\|_1 \\
    &\leq \| B - B^{(k)} \|_\infty + \left\| \mathbf{B}_i^{(k)}-\boldsymbol{\mathcal{B}}_i^{(k)} \right\|_1 + \| \mathcal{B}^{(k)} - \mathcal{R} \|_\infty.
\end{align*}
We have already computed $\left\| B - B^{(k)} \right\|_\infty \leq \ell (1-d)^k/d$, and the last norm can be computed as follows:
\begin{align*}
    \left\| \mathcal{B}^{(k)}-\mathcal{R} \right\|_\infty &\leq \sum_{t=k}^{\infty} (1-c-d)^t \| W^{(t)} \|_\infty +  \sum_{t=k}^{\infty} \sum_{s=1}^t a_{s,t} \left\| M^s \bar W  \right\|_\infty \\
&\leq \ell  \sum_{t=k}^{\infty} (1-c-d)^t  +  \ell \sum_{t=k}^{\infty} \sum_{s=1}^t a_{s,t}   = \ell  \sum_{t=k}^{\infty} (1-c-d)^t  + \ell \sum_{t=k}^\infty \sum_{s=1}^t \binom{t}{s} (1-c-d)^{t-s} c^s  \\
&= \ell  \sum_{t=k}^{\infty} (1-c-d)^t  +  \ell \sum_{t=k}^{\infty} ((1-d)^t  - (1-c-d)^t) = \ell \sum_{t=k}^\infty (1-d)^t = \ell (1-d)^k/d,
\end{align*}
where in the second inequality we used the observation that $\| W^{(t)} \|_\infty \leq \ell$ and $\| M^s \bar W \|_\infty \leq \ell$.

We have thus shown that for any $i \in V_n$ and any $k \geq 0$,
\[  \left\| \mathbf{B}_i-\boldsymbol{\mathcal{R}}_i \right\|_1 \leq \left\| \mathbf{B}_i^{(k)}-\boldsymbol{\mathcal{B}}_i^{(k)} \right\|_1 + 2\ell (1-d)^k/d.  \]
Taking expectations on both sides, followed by the maximum over $i$, we get
\begin{equation}\label{sup_max_to0}
    \max_{i\in V_n}\mathbb{E}_n\left[ \left\| \mathbf{B}_i-\boldsymbol{\mathcal{R}}_i \right\|_1 \right]\leq 2\ell (1-d)^k/d+\sup_{m\geq0}\max_{i\in V_n}\mathbb{E}_n\left[ \left\| \mathbf{B}_i^{(m)}-\boldsymbol{\mathcal{B}}_i^{(m)} \right\|_1 \right].
\end{equation}
Finally, note that by Theorem~\ref{MFA_generalthm}, 
\begin{align*}
\sup_{m\geq0}\max_{i\in V_n}\mathbb{E}_n\left[ \left\| \mathbf{B}_i^{(m)}-\boldsymbol{\mathcal{B}}_i^{(m)} \right\|_1 \right]  &= \sup_{m\geq0}\max_{i\in V_n}\mathbb{E}_n\left[ \left\| \mathbf{R}_i^{(m)}-\boldsymbol{\mathcal{R}}_i^{(m)} \right\|_1 \right]  \xrightarrow{P} 0, \qquad n \to \infty.
\end{align*}
Taking $k\to\infty$ in (\ref{sup_max_to0}) completes the proof.  
\end{proof}

\subsection*{Acknowledgments}
 This work was supported by NSF grant CMMI-2243261.

\bibliographystyle{apalike}
\bibliography{references}

\begin{thebibliography}{}

\bibitem[Abramowitz and Saunders, 2008]{abramowitz2008polarization}
Abramowitz, A.~I. and Saunders, K.~L. (2008).
\newblock Is polarization a myth?
\newblock {\em The Journal of Politics}, 70(2):542--555.

\bibitem[Aldous and Steele, 2004]{aldous2004objective}
Aldous, D. and Steele, J.~M. (2004).
\newblock The objective method: probabilistic combinatorial optimization and
  local weak convergence.
\newblock {\em Probability on discrete structures}, pages 1--72.

\bibitem[Avrachenkov et~al., 2018]{avrachenkov2018mean}
Avrachenkov, K., Kadavankandy, A., and Litvak, N. (2018).
\newblock Mean field analysis of personalized {P}age{R}ank with implications
  for local graph clustering.
\newblock {\em Journal of statistical physics}, 173:895--916.

\bibitem[Axelrod, 1997]{axelrod1997dissemination}
Axelrod, R. (1997).
\newblock The dissemination of culture: A model with local convergence and
  global polarization.
\newblock {\em Journal of conflict resolution}, 41(2):203--226.

\bibitem[Banerjee and Olvera-Cravioto, 2022]{banerjee2022pagerank}
Banerjee, S. and Olvera-Cravioto, M. (2022).
\newblock Pagerank asymptotics on directed preferential attachment networks.
\newblock {\em The Annals of Applied Probability}, 32(4):3060--3084.

\bibitem[Ben-Naim et~al., 2003]{ben2003bifurcations}
Ben-Naim, E., Krapivsky, P.~L., and Redner, S. (2003).
\newblock Bifurcations and patterns in compromise processes.
\newblock {\em Physica D: nonlinear phenomena}, 183(3-4):190--204.

\bibitem[Benjamini and Schramm, 2001]{10.1214/EJP.v6-96}
Benjamini, I. and Schramm, O. (2001).
\newblock {Recurrence of Distributional Limits of Finite Planar Graphs}.
\newblock {\em Electronic Journal of Probability}, 6(none):1 -- 13.

\bibitem[Berger, 1981]{berger1981necessary}
Berger, R.~L. (1981).
\newblock A necessary and sufficient condition for reaching a consensus using
  {D}e{G}root's method.
\newblock {\em Journal of the American Statistical Association},
  76(374):415--418.

\bibitem[Bessi et~al., 2016]{Bessi_etal_16}
Bessi, A., Zollo, F., Del~Vicario, M., Puliga, M., Scala, A., Caldarelli, G.,
  Uzzi, B., and Quattrociocchi, W. (2016).
\newblock Users polarization on {F}acebook and {Y}ou{T}ube.
\newblock {\em PLOS ONE}, 11(8):1--24.

\bibitem[Bhamidi et~al., 2019]{Bha_Bud_Wu_19}
Bhamidi, S., Budhiraja, A., and Wu, R. (2019).
\newblock Weakly interacting particle systems on inhomogeneous random graphs.
\newblock {\em Stochastic Processes and their Applications}, 129(6):2174--2206.

\bibitem[Budhiraja et~al., 2019]{Bud_Muk_Wu_19}
Budhiraja, A., Mukherjee, D., and Wu, R. (2019).
\newblock Supermarket model on graphs.
\newblock {\em The Annals of Applied Probability}, 29(3):1740 -- 1777.

\bibitem[Chaintron and Diez, 2022a]{chaintron2021propagation}
Chaintron, L.-P. and Diez, A. (2022a).
\newblock Propagation of chaos: A review of models, methods and applications.
  {I}. {M}odels and methods.
\newblock {\em Kinetic and Related Models}, 15(6):895--1015.

\bibitem[Chaintron and Diez, 2022b]{chaintron2022propagationapplications}
Chaintron, L.-P. and Diez, A. (2022b).
\newblock Propagation of chaos: A review of models, methods and applications.
  {II}. {A}pplications.
\newblock {\em Kinetic and Related Models}, 15(6):1017--1173.

\bibitem[Chen et~al., 2014]{chen2014pagerank}
Chen, N., Litvak, N., and Olvera-Cravioto, M. (2014).
\newblock Pagerank in scale-free random graphs.
\newblock In {\em Algorithms and Models for the Web Graph: 11th International
  Workshop, WAW 2014, Beijing, China, December 17-18, 2014, Proceedings 11},
  pages 120--131. Springer.

\bibitem[Cocomello and Ramanan, 2023]{cocomello2023exact}
Cocomello, J. and Ramanan, K. (2023).
\newblock Exact description of limiting {SIR} and {SEIR} dynamics on locally
  tree-like graphs.
\newblock {\em arXiv preprint arXiv:2309.08829}.

\bibitem[Como and Fagnani, 2011]{8c5aa4c6-94cc-393f-a938-57db54441872}
Como, G. and Fagnani, F. (2011).
\newblock Scaling limits for continuous opinion dynamics systems.
\newblock {\em The Annals of Applied Probability}, 21(4):1537--1567.

\bibitem[Dandekar et~al., 2013]{dandekar2013biased}
Dandekar, P., Goel, A., and Lee, D.~T. (2013).
\newblock Biased assimilation, homophily, and the dynamics of polarization.
\newblock {\em Proceedings of the National Academy of Sciences},
  110(15):5791--5796.

\bibitem[Deffuant et~al., 2002]{deffuant2002can}
Deffuant, G., Amblard, F., Weisbuch, G., and Faure, T. (2002).
\newblock How can extremism prevail? {A} study based on the relative agreement
  interaction model.
\newblock {\em Journal of artificial societies and social simulation}, 5(4).

\bibitem[Deffuant et~al., 2000]{deffuant2000mixing}
Deffuant, G., Neau, D., Amblard, F., and Weisbuch, G. (2000).
\newblock Mixing beliefs among interacting agents.
\newblock {\em Advances in Complex Systems}, 3(01n04):87--98.

\bibitem[DeGroot, 1974]{degroot1974reaching}
DeGroot, M.~H. (1974).
\newblock Reaching a consensus.
\newblock {\em Journal of the American Statistical association},
  69(345):118--121.

\bibitem[Del~Vicario et~al., 2017]{del2017modeling}
Del~Vicario, M., Scala, A., Caldarelli, G., Stanley, H.~E., and Quattrociocchi,
  W. (2017).
\newblock Modeling confirmation bias and polarization.
\newblock {\em Scientific reports}, 7(1):40391.

\bibitem[Fennell et~al., 2021]{fennell2021generalized}
Fennell, S.~C., Burke, K., Quayle, M., and Gleeson, J.~P. (2021).
\newblock Generalized mean-field approximation for the {D}effuant opinion
  dynamics model on networks.
\newblock {\em Physical Review E}, 103(1):012314.

\bibitem[Fraiman et~al., 2023]{fraiman2022stochastic}
Fraiman, N., Lin, T.-C., and Olvera-Cravioto, M. (2023).
\newblock Stochastic recursions on directed random graphs.
\newblock {\em Stochastic Processes and their Applications}, 166:104055.

\bibitem[Fraiman et~al., 2024]{fraiman2022opinion}
Fraiman, N., Lin, T.-C., and Olvera-Cravioto, M. (2024).
\newblock Opinion dynamics on directed complex networks.
\newblock {\em Mathematics of Operations Research (To appear)}.
\newblock arXiv preprint arXiv:2209.00969.

\bibitem[French, 1956]{French56}
French, J.~R. (1956).
\newblock A formal theory of social power.
\newblock {\em Psychological review}, 63(3):181--194.

\bibitem[Friedkin and Johnsen, 1990]{friedkin1990social}
Friedkin, N.~E. and Johnsen, E.~C. (1990).
\newblock Social influence and opinions.
\newblock {\em Journal of Mathematical Sociology}, 15(3-4):193--206.

\bibitem[Ganguly and Ramanan, 2022a]{ganguly2022hydrodynamic}
Ganguly, A. and Ramanan, K. (2022a).
\newblock Hydrodynamic limits of non-{M}arkovian interacting particle systems
  on sparse graphs.
\newblock {\em arXiv preprint arXiv:2205.01587}.

\bibitem[Ganguly and Ramanan, 2022b]{ganguly2022interacting}
Ganguly, A. and Ramanan, K. (2022b).
\newblock Interacting jump processes preserve semi-global {M}arkov random
  fields on path space.
\newblock {\em arXiv preprint arXiv:2210.09253}.

\bibitem[Garavaglia et~al., 2020]{garavaglia2020local}
Garavaglia, A., van~der Hofstad, R., and Litvak, N. (2020).
\newblock Local weak convergence for {P}age{R}ank.
\newblock {\em Annals of Applied Probability}.

\bibitem[Gentzkow, 2016]{gentzkow2016polarization}
Gentzkow, M. (2016).
\newblock Polarization in 2016.
\newblock {\em Toulouse Network for Information Technology Whitepaper}, 1.

\bibitem[G{\'o}mez-Serrano et~al., 2012]{gomez2012bounded}
G{\'o}mez-Serrano, J., Graham, C., and Le~Boudec, J.-Y. (2012).
\newblock The bounded confidence model of opinion dynamics.
\newblock {\em Mathematical Models and Methods in Applied Sciences},
  22(02):1150007.

\bibitem[Hetherington, 2009]{hetherington2009putting}
Hetherington, M.~J. (2009).
\newblock Putting polarization in perspective.
\newblock {\em British Journal of Political Science}, 39(2):413--448.

\bibitem[Ho{\l}yst et~al., 2001]{holyst2001social}
Ho{\l}yst, J.~A., Kacperski, K., and Schweitzer, F. (2001).
\newblock Social impact models of opinion dynamics.
\newblock {\em Annual Reviews Of Computational PhysicsIX}, pages 253--273.

\bibitem[Jia et~al., 2015]{jia2015opinion}
Jia, P., MirTabatabaei, A., Friedkin, N.~E., and Bullo, F. (2015).
\newblock Opinion dynamics and the evolution of social power in influence
  networks.
\newblock {\em SIAM review}, 57(3):367--397.

\bibitem[Lacker et~al., 2023a]{lacker2019local}
Lacker, D., Ramanan, K., and Wu, R. (2023a).
\newblock {Local weak convergence for sparse networks of interacting
  processes}.
\newblock {\em The Annals of Applied Probability}, 33(2):843 -- 888.

\bibitem[Lacker et~al., 2023b]{lacker2020marginal}
Lacker, D., Ramanan, K., and Wu, R. (2023b).
\newblock Marginal dynamics of interacting diffusions on unimodular
  {G}alton-{W}atson trees.
\newblock {\em Probability Theory and Related Fields}, 187:817--884.

\bibitem[Le~Boudec et~al., 2007]{le2007generic}
Le~Boudec, J.-Y., McDonald, D., and Mundinger, J. (2007).
\newblock A generic mean field convergence result for systems of interacting
  objects.
\newblock In {\em Fourth international conference on the quantitative
  evaluation of systems (QEST 2007)}, pages 3--18. IEEE.

\bibitem[Lee and Olvera-Cravioto, 2020]{lee2020pagerank}
Lee, J. and Olvera-Cravioto, M. (2020).
\newblock Page{R}ank on inhomogeneous random digraphs.
\newblock {\em Stochastic Processes and their Applications}, 130(4):2312--2348.

\bibitem[Liggett, 1999]{liggett1999stochastic}
Liggett, T.~M. (1999).
\newblock {\em Stochastic interacting systems: contact, voter and exclusion
  processes}, volume 324.
\newblock springer science \& Business Media.

\bibitem[Meng et~al., 2018]{meng2018opinion}
Meng, X.~F., Van~Gorder, R.~A., and Porter, M.~A. (2018).
\newblock Opinion formation and distribution in a bounded-confidence model on
  various networks.
\newblock {\em Physical Review E}, 97(2):022312.

\bibitem[Mynatt et~al., 1977]{mynatt1977confirmation}
Mynatt, C.~R., Doherty, M.~E., and Tweney, R.~D. (1977).
\newblock Confirmation bias in a simulated research environment: An
  experimental study of scientific inference.
\newblock {\em Quarterly Journal of Experimental Psychology}, 29(1):85--95.

\bibitem[Nickerson, 1998]{nickerson1998confirmation}
Nickerson, R.~S. (1998).
\newblock Confirmation bias: A ubiquitous phenomenon in many guises.
\newblock {\em Review of general psychology}, 2(2):175--220.

\bibitem[Olvera-Cravioto, 2022]{olvera2022strong}
Olvera-Cravioto, M. (2022).
\newblock Strong couplings for static locally tree-like random graphs.
\newblock {\em Journal of Applied Probability}, 59(4):1261--1285.

\bibitem[Parsegov et~al., 2016]{parsegov2016novel}
Parsegov, S.~E., Proskurnikov, A.~V., Tempo, R., and Friedkin, N.~E. (2016).
\newblock Novel multidimensional models of opinion dynamics in social networks.
\newblock {\em IEEE Transactions on Automatic Control}, 62(5):2270--2285.

\bibitem[Rainer and Krause, 2002]{Rainer2002-RAIODA}
Rainer, H. and Krause, U. (2002).
\newblock Opinion dynamics and bounded confidence: Models, analysis and
  simulation.
\newblock {\em Journal of Artificial Societies and Social Simulation}, 5(3).

\bibitem[Redner, 2019]{redner2019reality}
Redner, S. (2019).
\newblock Reality-inspired voter models: A mini-review.
\newblock {\em Comptes Rendus Physique}, 20(4):275--292.

\bibitem[Schmidt et~al., 2017]{Schmidt_etal_17}
Schmidt, A.~L., Zollo, F., Vicario, M.~D., Bessi, A., Scala, A., Caldarelli,
  G., Stanley, H.~E., and Quattrociocchi, W. (2017).
\newblock Anatomy of news consumption on {F}acebook.
\newblock {\em Proceedings of the National Academy of Sciences},
  114(12):3035--3039.

\bibitem[Sood and Redner, 2005]{sood2005voter}
Sood, V. and Redner, S. (2005).
\newblock Voter model on heterogeneous graphs.
\newblock {\em Physical review letters}, 94(17):178701.

\bibitem[Stroud, 2008]{Stroud_08}
Stroud, N.~J. (2008).
\newblock Media use and political predispositions: Revisiting the concept of
  selective exposure.
\newblock {\em Political Behavior}, 30(3):341--366.

\bibitem[Tian et~al., 2021]{tian2021social}
Tian, Y., Jia, P., Mirtabatabaei, A., Wang, L., Friedkin, N.~E., and Bullo, F.
  (2021).
\newblock Social power evolution in influence networks with stubborn
  individuals.
\newblock {\em IEEE Transactions on Automatic Control}, 67(2):574--588.

\bibitem[{Van Bavel} et~al., 2021]{VANBAVEL2021913}
{Van Bavel}, J.~J., Rathje, S., Harris, E., Robertson, C., and Sternisko, A.
  (2021).
\newblock How social media shapes polarization.
\newblock {\em Trends in Cognitive Sciences}, 25(11):913--916.

\bibitem[Van Der~Hofstad, 2024]{van2024random2}
Van Der~Hofstad, R. (2024).
\newblock {\em Random graphs and complex networks}, volume~2 of {\em Cambridge
  Series in Statistical and Probabilistic Mathematics}.
\newblock Cambridge University Press.

\bibitem[Weisbuch and Boudjema, 1999]{weisbuch1999dynamical}
Weisbuch, G. and Boudjema, G. (1999).
\newblock Dynamical aspects in the adoption of agri-environmental measures.
\newblock {\em Advances in Complex Systems}, 2(01):11--36.

\bibitem[Yang et~al., 2017]{Yang2017InnovationGE}
Yang, M., Qu, X., Cao, Z., and Yang, X. (2017).
\newblock Innovation governs everything eventually: Extensions of the
  {D}e{G}root model.
\newblock {\em Acta Mathematicae Applicatae Sinica, English Series}, 33:35--42.

\end{thebibliography}

\end{document}